\documentclass[12pt,letterpaper,titlepage]{amsart}
\usepackage{amsmath, amssymb, amsthm, amsfonts,amscd,amsaddr,wasysym}

\theoremstyle{plain}
\newtheorem{theorem}{Theorem}[section]

\newtheorem{cor}[theorem]{Corollary}

\newtheorem{prop}[theorem]{Proposition}
\newtheorem{lemma}[theorem]{Lemma}
\newtheorem{definition}[theorem]{Definition}
\newtheorem{problem}[theorem]{Problem}
\theoremstyle{definition}
\newtheorem{ex}[theorem]{Example}
\newtheorem{exc}[theorem]{Exercise}
\newtheorem{rmk}[theorem]{Remark}
\numberwithin{equation}{section}
\newtheorem*{theoremA*}{Theorem A}
\newtheorem*{theoremB*}{Theorem B}
\newtheorem*{theorem1*}{Theorem A'}
\newtheorem*{theoremC*}{Theorem C}
\newtheorem*{theoremD*}{Theorem D}
\newtheorem*{theoremE*}{Theorem E}
\newtheorem*{theoremF*}{Theorem F}
\newtheorem*{theoremE2*}{Theorem E2}
\newtheorem*{theoremE3*}{Theorem E3}
\newcommand{\bs}{\backslash}
\newcommand{\Cc}{\mathcal{C}}
\newcommand{\C}{\mathbb{C}}

\newcommand{\Lc}{\mathcal{L}}
\newcommand{\E}{\mathcal{E}}
\newcommand{\Hc}{\mathcal{H}}

\newcommand{\Q}{\mathbb{Q}}

\newcommand{\Z}{\mathbb{Z}}
\newcommand{\Zc}{\mathcal{Z}}

\newcommand{\Sc}{\mathcal{S}}

\newcommand{\R}{\mathbb{R}}
\newcommand{\N}{\mathbb{N}}

\newcommand{\Gr}{\operatorname{Gr}}
\newcommand{\Sl}{\operatorname{SL}}
\newcommand{\Ind}{\operatorname{Ind}}
\newcommand{\SO}{\operatorname{SO}}
\newcommand{\Hom}{\operatorname{Hom}}
\newcommand{\End}{\operatorname{End}}

\newcommand{\GL}{\operatorname{GL}}

\newcommand{\tr}{\operatorname{tr}}

\newcommand{\Ad}{\operatorname{Ad}}

\newcommand{\ad}{\operatorname{ad}}
\newcommand{\diag}{\operatorname{diag}}

\newcommand{\vol}{\operatorname{vol}}
\newcommand{\Res} {\operatorname{Res}} 

\newcommand{\supp}{\operatorname{supp}}
\newcommand{\Span}{\operatorname{span}}

\newcommand{\rank}{\operatorname{rank}}

\newcommand{\re}{\operatorname{Re}}

\def\hat{\widehat}
\def\af{\mathfrak{a}}

\def\e{\epsilon}
\def\gf{\mathfrak{g}}

\def\cf{\mathfrak{c}}

\def\hf{\mathfrak{h}}
\def\kf{\mathfrak{k}}
\def\lf{\mathfrak{l}}
\def\mf{\mathfrak{m}}
\def\nf{\mathfrak{n}}

\def\pf{\mathfrak{p}}

\def\qf{\mathfrak{q}}

\def\sf{\mathfrak{s}}

\def\gl{\mathfrak{gl}}

\def\so{\mathfrak{so}}

\def\uf{\mathfrak{u}}

\def\zf{\mathfrak{z}}
\def\la{\langle}
\def\ra{\rangle}
\def\1{{\bf1}}

\def\U{\mathcal{U}}

\def\Cc{\mathcal{C}}
\def\D{\mathcal {D}}

\def\M{\mathcal{M}}
\def\oline{\overline}
\def\F{\mathcal{F}}

\def\W{\mathcal{W}}

\def\tilde{\widetilde}

\def\v{\mathbf{v}}

\usepackage[usenames]{color}

\title[Sanya lectures]
{Harmonic analysis for real spherical spaces}

\begin{document}
\date{March 22, 2017}

\begin{abstract} 
We give an introduction to basic harmonic analysis and representation theory for homogeneous spaces 
$Z=G/H$ attached to a real reductive Lie 
group $G$. A special emphasis is made to the case where $Z$ is real spherical.

\end{abstract}

\author[Kr\"otz]{Bernhard Kr\"{o}tz}
\email{bkroetz@gmx.de}
\address{Universit\"at Paderborn, Institut f\"ur Mathematik\\Warburger Stra\ss e 100, 
33098 Paderborn}
\author[Schlichtkrull]{Henrik Schlichtkrull}
\email{schlicht@math.ku.dk}
\address{University of Copenhagen, Department of Mathematics\\Universitetsparken 5, 
DK-2100 Copenhagen \O, Denmark}

\maketitle
\section{Introduction}

It is a well-known fact that in the decomposition
of the coordinate ring of a reductive complex algebraic group 
$G$ each irreducible algebraic representation occurs
with multiplicity exactly equal to its dimension. 
In fact, 
\begin{equation} \C[G]= \oplus_{\lambda\text{ irrep}} \C[G]^\lambda
\simeq \oplus_{\lambda\text{ irrep}} 
V_\lambda\otimes V_\lambda^*
\end{equation} 
as a left$\times$right representation of $G\times G$
(the sum is over equivalence classes of
irreducible representations). 
This result is closely related (via the so-called
Weyl's unitary trick) to the Peter-Weyl 
theorem for compact Lie groups, which exhibits
a similar decomposition of the Hilbert space $L^2(G)$.

For
$Z=G/H$, a homogeneous space of an algebraic subgroup
$H\subset G$, it immediately follows that
\begin{equation} \label{PW1}\C[Z]\simeq \oplus_{\lambda\text{ irrep}} 
V_\lambda\otimes V_\lambda^{*H},\end{equation}
where $V_\lambda^{*H}\subset V_\lambda^*$ denotes the space of 
$H$-fixed
linear forms on $V_\lambda$. In particular, the multiplicity
of an equivalence class $\lambda$ in $\C[Z]$ is
exactly the dimension of $V_\lambda^{*H}$.
Again an analogous decomposition holds for the 
$L^2$-Hilbert space of a
homogeneous space of a compact Lie group, and it represents
the epitome of {\it harmonic analysis} on such spaces.

When it comes to homogeneous spaces of non-compact reductive
Lie groups the harmonic analysis is however much more 
complicated. A major complication comes already from the fact
that irreducible representations are infinite dimensional
in general. The decomposition of $L^2(G)$,
the so-called {\it Plancherel formula},
for a reductive real Lie group was the life time
achievement of Harish-Chandra, and the generalization
to homogeneous spaces of $G$ is still far from being
understood. However, recent investigations 
(see \cite{KKSS2}, \cite{KKS2}) suggest that 
if the homogeneous space $G/H$ is {\it real spherical}, 
then such a decomposition may be reachable.  

\par It is the aim of these notes to explain some of the
background for this development for readers who
are not well acquainted with the representation
theory of real reductive Lie groups. In particular,
we will emphasize the {\it geometric} properties
of real spherical spaces that are relevant. The main results
are taken from \cite{B} and \cite{KS}, \cite{KKSS2}.

\par Let us now be a bit more specific about the contents of this survey. 
In general, for a homogeneous space $Z=G/H$ with $G$-invariant measure one has 
an abstract Plancherel formula for the left regular representation of $G$ on $L^2(Z)$: 

\begin{equation} \label{PW2}  L^2(Z) \simeq  \int_{\hat G}^\oplus \Hc_\pi \otimes {\mathcal M}_\pi 
\ d\mu(\pi) .\end{equation} 
Here $\hat G$ is the unitary dual of $G$, the set of equivalence classes of irreducible unitary representations of $G$. 
This set carries a natural topology, and $\mu$ is a Radon measure on $\hat G$, the so-called {\it Plancherel measure of $Z$}.
In case $G$ is compact, we recall that  $\hat G$ is discrete and  
taking $G$-finite vectors in (\ref{PW2}) yields
(\ref{PW1}). The multiplicity space ${\mathcal M}_\pi$ corresponds to $(V_\lambda^*)^H$ in (\ref{PW1}), but
in general it is not $(\Hc_\pi^*)^H$ or a subspace theoreof. It is however 
contained in an enlarged dual $(\Hc_\pi^{-\infty})^H$, the   space of 
$H$-invariant distribution vectors of the representation $(\pi,\Hc_\pi)$.   To be more precise, $\Hc_\pi^{-\infty}$ is 
the strong dual of the space of smooth vectors $\Hc_\pi^\infty$ of the representation $(\pi,\Hc_\pi)$. 
As a topological vector space $\Hc_\pi^\infty$ is a so-called Fr\'echet space. 
\par This article starts with a survey on topological vector spaces with an emphasis on Fr\'echet spaces,
and continues in Section 3
with a review of group representations on topological vector spaces. In particular we explain smooth vectors.  The 
representation of our concern is the left regular representation of $G$ on $L^2(Z)$.   The characterization of its smooth vectors
$L^2(Z)^\infty$ is tied to volume growth on the homogeneous space $Z$. This is the topic of Section 4. The main tool is the 
invariant Sobolev lemma of Bernstein \cite{B} of which we include a proof.  
From Section 5 on we are more specific about 
our group $G$ and request that is real reductive. We review the basic 
representation theory of infinite dimensional representations, i.e.~the theory of Harish-Chandra modules and their smooth completions.  Beginning with Section 6  we consider 
homogeneous spaces $Z=G/H$ for $H<G$ a closed subgroup. Our first topic is of
generalized matrix coefficients for a representation $(\pi, E)$: attached to a smooth vector $v\in E^\infty$ and 
$\eta \in (E^{-\infty})^H$ we form the function $m_{v,\eta}(g)= \eta (\pi(g^{-1}) v)$ which descends to a smooth function on 
$Z=G/H$.   This gives us a large supply of smooth functions on $Z$.  Smooth Frobenius reciprocity asserts that 
$(E^{-\infty})^H$ parametrizes $\Hom_G(E^\infty, C^\infty(Z))$, i.e. the space of continuous  $G$-equivariant embeddings of $E^\infty$ 
into $C^\infty(Z)$. 
\par In Section 7 we introduce homogeneous real spherical spaces and review the local structure theorem (LST).  The LST allows us 
to attach certain geometric invariants to a real spherical space, the real rank and a conjugacy class of a parabolic 
subgroup $Q$ of $G$, called $Z$-adapted. These geometric invariants are tied to representation theory as follows:
one is led to consider spherical pairs $(V,\eta)$, that is, $V$ is a Harish-Chandra module and $\eta$ is a non-zero continuous
$H$-invariant linear functional on $V^\infty$, the unique smooth moderate growth completion of $V$. The spherical subrepresentation theorem then asserts that irreducible spherical pairs admit embedding in induced modules 
for the parabolic $Q$. In addition there are sharp bounds for the 
dimension of the space $(V^{-\infty})^H$; in particular, the 
multiplicity spaces ${\mathcal M}_\pi$ from above are finite dimensional for a real spherical space. 

\par In Section 8 we give a short introduction to direct integrals of Hilbert spaces and explain the abstract 
Plancherel decomposition (\ref{PW2}).    The main goal of $L^2$-harmonic analysis on $Z$ is the determination 
of the Plancherel measure $\mu$.  The first step is to determine the support $\supp \mu \subset \hat G$, a bit 
more precisely one should determine all spherical pairs $(V_\pi, \eta)$ where $V_\pi$ is the Harish-Chandra module 
of $\Hc_\pi$, $\eta\in {\mathcal M}_\pi$ and $\pi \in \supp(\mu)$.  Those pairs are called {\it tempered}.

\par It was realized by Bernstein, in a very general setup, that the determination of all tempered real spherical pairs
is tied to a certain Schwartz space $\Cc(Z)$ of smooth rapidly decreasing functions on $Z$.  In general
the definition of $\Cc(Z)$ is based on the volume growth of $Z$. For a real spherical space this growth 
can be exactly determined  via the polar decomposition of $Z$. This yields us very explicit families of semi-norms which determine the topology 
on the Fr\'echet space $\Cc(Z)$.  All that is explained 
in Section 9. 

\par Our survey ends with a growth bound for generalized matrix coefficients for a tempered spherical pair
$(V,\eta)$.  Let us mention that this bound is the starting bound for the characterization of the tempered spectrum of $Z$ in terms 
of the relative discrete spectrum of certain satellites of $Z$, called boundary degenerations (see \cite{KKS}).
As for harmonic analysis on real spherical spaces this is the current state of the art.

\par For $p$-adic spherical spaces the theory is further developed and a Plancherel theorem is established 
for an interesting class of spherical spaces, termed wave-front. 
This is the work of Sakellaridis and Venkatesh, who in particular developed new geometric 
concepts for Plancherel theory \cite{SV}.  Presently it appears that part of their geometric ideas can be adapted to 
real spherical  spaces, but the analytic aspects of the discrete spectrum are not parallel. It is an exciting 
topic for future research.

\smallskip 
These lecture notes were originally prepared for a mini-course given at 
the Tsinghua Sanya International Mathematics Forum in November 2016.
We are greatful to Michel Brion and Baohua Fu for providing the
opportunity to present the material on that occasion.

\section{Topological vector spaces}

\subsection{Generalities}
Let $K$ be a field and $E$ be a $K$-vector space.  This means that we have two structures:

\bigskip 
\par i)  {\it scalar multiplication}: 
$$ {\rm sc}:  K \times E\to E, \ \ (\lambda, x)\mapsto \lambda  x, $$

\par ii) {\it addition}:
$$ {\rm add}: E \times E \to E, \ \ (x,y)\mapsto x+y\, .$$

This gives a category ${\rm \bf Vect}_K$ with morphisms the $K$-linear maps $T: E_1 \to E_2$. 
Next I want to explain 

\bigskip 
\begin{center}\framebox{Topological linear algebra = functional analysis}\end{center}

\bigskip

The easiest vector space is $E=K$.   
We call $K$ a topological field if it 
is endowed with a topology such that 
for $E=K$ the operations ${\rm sc}$ and {\rm add} are
continuous, and, in addition, 
the inversion on $K^\times=K\bs\{0\}$ is continuous. 
The important examples are the local fields
$K=\R,\C, \Q_p, {\mathbb F}_q$
and ${\mathbb F}_q(X)$.  

\par From now on we assume that $K$ is a topological field.   
A {\it topological vector space}
over $K$  is a $K$-vector space $E$ 
which is endowed with a topology such that the structure operations ${\rm sc}$ and ${\rm add}$ become continuous. 
This gives us a new category ${\rm \bf TopVect_K}$ with morphisms the continuous  $K$-linear maps. 

\par The theory is sensitive on the nature of the topological field.   For example there is no notion of convexity if ${\rm char} K>0$. 
From now on we assume that $K=\C$.

We request that $E$ is separated (Hausdorff) which is equivalent to the fact that $\{0\}$ is closed.   This already brings 
us a variety of problems. 

\par Let $E$ be a TVS (topological vector space) and $F\subset E$ a subspace.  Then:

\begin{itemize} 
\item $F$ endowed with its subspace topology is a TVS.  Further, $F$ is separated.  
\item $E/F$ endowed with the quotient topology is a TVS.  However, $E/F$ is separated iff $F$ is closed. 
\end{itemize}

So we need to be careful with quotients, especially because the most interesting
morphisms of our category rarely have closed images.
Note that non-separated spaces typically have little meaning.  

Also we request from now on that our TVS are {\it separable}, that is there exists a countable dense subset.

\subsection{Examples of TVS}

\subsubsection{Hilbert spaces}
Since all (separable) Hilbert spaces admit a countable orthonormal basis they are unitarily isomorphic to either $\C^n$ or 
$$\ell^2(\N_0)=\{ x=(x_n)_{n\in\N_0}\mid x_n\in \C, \    \la x, x\ra:=\sum_{n=0}^\infty |x_n|^2<\infty\}\, .$$

If $M$ is a manifold endowed with a positive Radon measure $\mu_M$, then $L^2(M, \mu_M)$ is a Hilbert space. 
Often it is quite difficult,  if not impossible to write down a concrete orthonormal basis for $L^2(M, \mu_M)$.  For example, 
if $M=\R$ and 
$\mu$ is the Lebesgue measure you will already find this a quite challenging task. One classical solution is 
given by the Hermite functions $H_n(x)= e^{-x^2/2} h_n(x)$ with $h_n$ the Hermite-polynomials. 

Later our concern will be with $E=L^2(G/H)$ where $G/H$ is a unimodular real spherical space. 

\par In the context of Hilbert spaces we can already see that morphisms of Hilbert spaces generically have non-closed 
image.   Consider 
$$T:\ell^2(\N_0) \to \ell^2(\N_0), \ \ (x_n)_{n\in\N_0} \mapsto \big(\frac{1}{n+1}  x_n\big)_{n\in \N_0}\, .$$
Then  $y^k:=(1, \frac{1}{2}, \frac{1}{3}, \ldots, \frac{1}{k}, 0, 0,\ldots)  \in \operatorname {im}  T$ for every $k\in\N$.
Moreover, the $(y^k)_{k\in\N}$ form a Cauchy-sequence with limit $y=\lim y^k\not\in \operatorname {im} T$.

\subsubsection{Banach spaces}  A vector space $E$ endowed with a norm $p$ is called a {\it Banach space} if $(E,p)$ is a 
complete topological space. Typical examples are $\ell^k(\N)$ for $1\leq k\leq \infty$ and, more generally 
$L^k(M)$ for any measure space $M$.  Others are 
\begin{align*}
C_b(M)&:=\{ f\in C(M)\mid f\  \text{bounded}\}\\ 
C_0(M)&:=\{ f\in C(M)\mid    \lim_{m\to \infty} f(m) =0\}
\end{align*}
for a locally compact space $M$, with the uniform norm
$$p(f):=\sup_{m\in M} |f(m)|.$$
Here $C(M)$ denotes the space of continuous functions,
and the condition of vanishing limit at infinity
means that $\{m\in M\mid |f(m)|\ge c\}$
is compact for all $c>0$.

\subsubsection{Fr\'echet spaces}  Contrary to what you might expect it is not the class of Banach spaces or Hilbert 
spaces which is most important for the discussion of infinite dimensional representations but another more flexible
class which we introduce next. 

\par Let $E$ be a vector space. A function $p: E \to \R_{\geq 0}$ is called a {\it semi-norm} provided that 
\begin{itemize} 
\item $p(\lambda x) = |\lambda| p(x)$ for all $x\in E$ and $\lambda\in\C$. 
\item $p(x+ y) \leq p(x) + p(y)$ for all $x,y\in E$.
\end{itemize}

Observe that we do not request $p(x)=0\Rightarrow x=0$. 
For example let $M$ be a locally compact space and $E=C(M)$.  For a compact subset $K\subset M$ set 
$$ p(f):=\sup_{m\in K} |f(m)|   \qquad (f\in E)\, .$$
Then $p$ is a norm if and only if  $K= M$.

\par A  semi-normed space $(E,p)$ is a typically non-separated TVS. However, it features a natural 
completion $E_p$ which is a Banach space.  Concretely 
$E_p$ can be constructed in two steps: i) Note that $F:=\{p=0\}$ is a closed subspace of $E$ and $\tilde E:=E/F$ is separated.  Further 
$p$ induces a norm $\tilde p$ on $\tilde E$ so that we obtain a normed space $(\tilde E, \tilde p)$;    ii)  The completion of 
the normed space $(\tilde E, \tilde p)$ then is the desired Banach space $E_p$. 

\par In our example $E=C(M)$ from above one has $E_p=C(K)$
(a consequence of the Tietze extension theorem).

 \par  Having introduced semi-normed spaces we can now define Fr\'echet spaces.   
 Let $E$ be a vector space endowed with a countable family of semi-norms $(p_n)_{n\in\N}$.  We endow $E$ with the coarsest 
 topology  such that the diagonal embedding 
 $$ E \to \prod_{n\in N} (E, p_n)$$
 becomes continuous (initial topology).  An alternative description of this topology is given via the semi-metric 
 $$d(x,y):=\sum_{n=1}^\infty {1\over 2^n}  {  p_n(x-y)\over 1 + p_n(x-y)}  \qquad (x,y\in E)$$

A vector space $E$ endowed with a countable family of semi-norms $(p_n)_{n\in N}$  is a {\it Fr\'echet space}  provided that $(E, d)$ is a complete 
metric space.  Note that $E$ is a TVS.  In applications the topology on $E$ can be introduced 
from different families of semi-norms $(q_n)_{n\in\N}$. To get a flavour, set 
$q_n:=p_1+\ldots +p_n$. Then the $(q_n)_n$ induce the same topology as the $(p_n)_n$. In particular, we 
can always assume that the family $(p_n)_n$ is increasing.

\subsubsection{Examples of Fr\'echet spaces} 
It is clear that Banach spaces are Fr\'echet spaces, but the converse is not true
as the following example shows.

\par {\it Schwartz spaces on $\R^m$.}   For every $n\in \N_0$ and 
every smooth function $f\in C^\infty(\R^m)$ set 
$$ p_n(f):=\sup_{x\in \R^m}  (1 + \|x\|)^n \sum_{\alpha\in \N_0^m\atop |\alpha|\leq n}|\partial^\alpha f(x)|\, .$$
Then 
$${\mathcal S}(\R^m):=\{ f\in C^\infty(\R^m) \mid \forall n\in \N, \ p_n(f)<\infty\}$$
is a Fr\'echet space, the so-called {\it Schwartz space} of smooth rapidly decreasing functions on $\R^m$. 

Later on our concern will be with $\Cc(G/H)$, the Harish-Chandra Schwartz space of a real spherical space.

\begin{exc} (a) Show that ${\mathcal S}(\R^m)$ has the Heine-Borel property: {\it closed and bounded 
sets are compact}.  Here a subset $B\subset E$ is called bounded provided that there exists for all 
$n\in\N$ a constant  $c_n>0$ such that $\sup_{v\in B} p_n(v)<c_n$.  Hint: Show that every sequence in 
a closed and bounded subset has a convergent subsequence.
\par (b) Conclude from (a) that ${\mathcal S}(\R^m)$ 
is not a Banach space. Hint: A basic theorem of Riesz asserts that a locally compact 
topological vector space is finite dimensional.
\end{exc}

 \par {\it Continuous functions.} A further typical example is $E=C(M)$ for a manifold $M$ with compact exhaustion $M=\bigcup_{n\in \N} K_n$
 and associated semi-norms $p_n(f):=\sup_{m \in K_n} |f(m)|$. The resulting topology is the one of locally uniform 
 convergence. 
 
\begin{exc} Define the natural Fr\'echet topology on $C^\infty(M)$ for a manifold $M$. 
\end{exc}
 
\subsection{Functional analysis}
Finally we come to the explanation why topological linear algebra is usually referred to 
as functional analysis. The theory emerged as a theory of function spaces and 
linear operators between them. The main examples of TVS are spaces of functions on manifolds, 
and many abstract notions in the theory are motivated by that. For example there is the
notion of a Montel space, which is derived from
Montel's theorem about families of holomorphic functions. 
The crowning achievement is perhaps Grothendieck's notion of nuclear 
Fr\'echet spaces which puts the Schwartz kernel theorem from (functional) analysis 
in a much wider and more transparent context.

\par We conclude with an issue of functional analysis which will concern us later on (see Problem \ref{c-conj} below).
We are interested in  linear differential operators, i.e.~operators which in their simplest form are 
of the type 
$$T:  C^\infty(U) \to C^\infty(U), \ \ f\mapsto   \sum_{|\alpha|\leq n}  a_\alpha \partial^\alpha f$$
where $U\subset \R^N$ is an open set and $a_\alpha$ are smooth coefficient functions.  
In particular we would like to know  whether such an operator is surjective or
at least has closed range. 
 Here is the famous Lewy example which destroys all hopes for general results.  Let $U=\R^3$ and 
$$T=\partial_x + i \partial_y  -  2i(x +iy)  \partial_t\, .$$
Then $ T(u)=f$ implies that $f$ is analytic in $t$ near $t=0$.  In particular, $T$ is neither surjective nor 
has closed range.

\section{Group representations}
Let $G$ be a locally compact group. 
For a topological vector space $E$ we set 
$$\GL(E):=\{ T: E \to E \mid T \ \text{linear, bijective; $T, T^{-1}$ continuous}\}$$ 
A group homomorphism $\pi: G\to \GL(E)$ is called a {\it representation}, provided that the associated 
action map 
$$  G \times E \to E, \ \ (g, v)\mapsto \pi(g)v$$
is continuous.  
\begin{rmk} This definition is not what you might expect.  For example,  if $E$ is a normed space, then 
$\Hom(E,E)$ is normed as well; if $p$ is the norm on $E$, then the operator norm is declared 
as 
$$\|T\| :=\sup_{p(v)\leq 1}  p(Tv)\qquad (T\in \Hom(E,E))$$

Now we could consider $\GL(E)\subset \Hom(E,E)$ endowed with the subspace topology and request that 
$\pi: G\to \GL(E)$ is continuous.  As we will see soon this is a too strong requirement which would exclude 
almost all interesting examples if $\dim E=\infty$.
\end{rmk}

By a {\it morphism}  between two $G$-representations $(\pi_1, E_1)$ and $(\pi_2,E_2)$ we understand 
a continuous linear map $T: E_1 \to E_2$ which intertwines the $G$-action, i.e. $T\circ \pi_1(g)= \pi_2(g) \circ T$ 
for all $g\in G$. 

\par If $E$ is a Banach, resp.~Fr\'echet,  space, then we call $(\pi,E)$ a {\it Banach, resp.~Fr\'echet,  representation}. An application of the 
uniform boundedness principle then gives the following useful criterion: 

\begin{lemma} Let $E$ be a Banach space and $\pi: G\to \GL(E)$ be a group homomorphism. 
Then the following are equivalent: 
\begin{enumerate}
\item $(\pi, E)$ is a representation. 
\item For all $v\in E$, the orbit map $f_v: G\to E,\ g\mapsto \pi(g)v$ is continuous. 
\end{enumerate}
\end{lemma}

\begin{proof} The implication (1) $\Rightarrow$ (2) is almost clear: If we restrict the continuous 
action map $G\times E \to E$, to $G\times\{v\}$ we obtain the orbit maps. 
\par For the converse implication let $p$ be the norm of $E$. We observe that continuity of $f_v$ implies that 
the map $G\ni g \mapsto  p(\pi(g) v)\in\R_{\geq 0}$ is continuous and hence locally bounded for all $v\in E$.

The uniform boundedness principle then implies that $g\mapsto \|\pi(g)\|$ is locally bounded. 

Now for $g'$ close to $g$ and $v'$ close to $v$ we obtain 
that 
\begin{eqnarray*} \pi(g)v - \pi(g')v'& =& (\pi(g)v - \pi(g')v)  + (\pi(g')v - \pi(g')v') \\ 
&=& (f_v(g) - f_v(g'))  + \pi(g') (v - v')\end{eqnarray*}
and thus $\pi(g)v$ is close to $\pi(g')v'$. 
\end{proof}

\begin{ex} (Exercise, left regular representation)  Let $G=(\R^n, +)$ and $E=L^p(\R^n)$ for $1\leq p\leq \infty$. 
\par (a) Show that 
$$L:  G \to \GL(E), \ \ x\mapsto L(x);\ L(x)f:= f(\cdot - x)$$
defines a group homomorphism with $\|L(x)\|=1$ for all $x\in G$. 
\par (b)  Show that $L: G\to \GL(E)$ is not continuous if $\GL(E)$ is endowed with the operator topology. 
\par (c) $(L,E)$ is a representation if and only if $p<\infty$.  (Hint: Use that $C_c(\R^n)$ is dense  in $E$ for $p<\infty$.)
\end{ex}

The above example admits the following generalization. Let $H<G$ be a closed subgroup and form the homogeneous
space $Z:=G/H$.  We assume that $Z$ is {\it unimodular}, that is, $Z$ carries a positive $G$-invariant Radon measure. 
This measure is then unique up to scalar (Haar measure).  We write $z_0=H$ for the standard base point of $Z$. 

\par  Then $E=L^p(G/H)$ with $1\leq p<\infty$ is a Banach space and 
$$ L : G \to \GL(E); \  L(g)f(z)= f(g^{-1}z) \qquad (f\in E, g\in G, z\in Z)$$
defines a Banach representation. 

\subsection{Smooth vectors}
From now on we assume that $G$ is a Lie group. 
Let $(\pi,E)$ be a representation and $v\in E$.  Call $v$  {\it smooth} provided that 
the vector valued orbit map 
$f_v : G \to E$ is smooth. We set 
$$E^\infty=\{ v\in E\mid v\  \text{is smooth}\}$$
and observe that $E^\infty$ is a $G$-invariant subspace of $E$. 

\begin{exc} Let $(\pi,E)$ be a representation on a complete TVS. Consider the test functions
$C_c^\infty(G)$ as an algebra under convolution:
$$(f_1 * f_2)  (g)=\int_G f_1 (x)  f_2(x^{-1} g )  \ dx $$ 
where $dx$ is a left Haar measure on $G$. 

\par (a) Show that $(\pi,E)$ gives rise to an algebra representation 
$$\Pi: C_c^\infty(G) \to \End(E)$$ 
where 

$$\Pi(f) v = \int_G f(x) \pi(x)v \ dx \qquad (f\in C_c^\infty(G), v \in E)\, .$$
(Hint: Use Riemann sums and the completeness of $E$ to show that the vector valued integrals converge
in $E$.)

\par (b)  Suppose now that $E$ is a Banach space.  Show that: 
\begin{enumerate}
\item $ \Pi(C_c^\infty(G))E\subset E^\infty$. 
\item $ E^\infty\subset E$ is dense. (Hint: Use a Dirac sequence $(\phi_n)_n$ of $G$ centered at $\1$.)
\end{enumerate} 

{\it Remark:} A much stronger statement than (1) is true, namely $E^\infty = \Pi(C_c^\infty(G))E$ (Theorem of 
Dixmier-Malliavin). 

\end{exc} 

From now on we adopt the general convention that when a Lie group is
denoted by an upper case Latin letter, then
the corresponding lower case German letter denotes the associated Lie algebra.
In particular $\gf:=\operatorname{Lie}(G)$ is the Lie algebra of $G$.  

For $X\in \gf$ and $v\in E^\infty$ 
the limit 
$$ d\pi(X)v:=\lim_{t\to 0}  \frac{\pi(\exp(tX))v - v}{t}$$
is defined and yields a Lie algebra representation 
$$d\pi: \gf \to \End(E^\infty) \, .$$
We let $\U(\gf)$ be the universal enveloping algebra of $\gf$ and extend $d\pi$ to an algebra representation 
of $\U(\gf)$.

Here is an example for $E^\infty$.  We let $E=L^p(G/H)$ as before.   Then, the local Sobolev lemma implies 
that 
$$E^\infty=\{ f\in C^\infty(Z)\mid \forall u \in \U(\gf) \ dL(u)f \in E\}\, .$$ 
In the next section we state a  version of the Sobolev lemma which takes also the geometry into account.

\subsection{The topology on smooth vectors}

Let $(\pi,E)$ be a Banach representation of $G$. Let $p$ be the norm on $E$ and 
fix a basis $X_1, \ldots, X_n$ of $\gf$.  Define for every 
$k\in \N_0$ a norm $p_k$ on $E^\infty$ by 
\begin{equation}\label{defi Sobolev norm}
p_k(v):=\sum_{\alpha\in\N_0^n\atop |\alpha|\leq k} p (d\pi(X_1^{\alpha_1}\cdot\ldots \cdot X_n^{\alpha_n}) v)
\end{equation}

We refer to $p_k$ as a $k$-th Sobolev norm of $p$. 

\begin{lemma} Let $(\pi,E)$ be a Banach representation of $G$. Then the following assertions hold:  
\begin{enumerate}
\item The $(p_k)_{k\in\N}$ define a Fr\'echet topology on $E^\infty$ and the topology is independent of the 
choice of the particular basis  $X_1, \ldots, X_n$ of $\gf$. 
\item The action $G\times E^\infty\to E^\infty$ is continuous and gives rise to a Fr\'echet representation 
$(\pi^\infty, E^\infty)$ of $G$. 
\end{enumerate}
\end{lemma}

\subsection{F and SF-representations}

Let $(\pi, E)$ be a Fr\'echet representation.  We say that $(\pi, E)$ is an {\it $F$-representation} 
provided that there exists a family of semi-norms $(p_n)_{n\in\N}$ which induces the 
topology on $E$ such that the actions
$$G\times (E,p_n)\to (E, p_n)$$
are continuous for all $n\in \N$. 

\begin{rmk} The notion of $F$-representations is equivalent to the perhaps more familiar notion 
of moderate growth Fr\'echet representations (see \cite{BK}, Section 2.3.1 and in particular Lemma 2.10).
\end{rmk}

 An $F$-representation is called an {\it $SF$-representation} if all orbit
maps are smooth. By a morphism between two $SF$-representations 
$(\pi_1,E_1)$ and $(\pi_2, E_2)$ we understand 
a morphism $T: E_1\to E_2$  such that for some fixed $k\in \N$ the induced linear maps 
$(E_1, p_{1,n}) \to  (E_2, p_{2, k+ n})$ are continuous for all $n\in \N$.

\begin{lemma} Let $(\pi, E)$ be a Banach representation of $G$. Then $(\pi^\infty, E^\infty)$ is an 
SF-representation w.r.t.  a family of Sobolev norms.
\end{lemma}

\section{Volume weights on $G/H$}

Let $Z=G/H$ be a unimodular homogeneous space as before.  By a {\it weight} on $Z$ we understand 
a locally bounded
function $w: Z \to \R_{>0}$  with the following property. For all compact 
subsets $\Omega\subset G$ there exists a constant $C>0$ such that 
\begin{equation}\label{defi weight}
 w(gz) \leq C w(z) \qquad (z\in Z, g\in \Omega)\, .
\end{equation}
For later reference, we note that by applying (\ref{defi weight}) to
the compact set
$\Omega^{-1}$ we obtain as well that there exists $C'>0$ such that
\begin{equation}\label{weight below}
 w(gz) \geq C' w(z) \qquad (z\in Z, g\in \Omega)\, .
\end{equation}

By a ball $B\subset G$ we understand a compact symmetric neighborhood of $\1$ in $G$. 
Fix a ball $B$ and define 
\begin{equation*}
{\bf v}(z):= \vol_Z(Bz)\qquad (z\in Z)\, .
\end{equation*}
We  refer to $\v$ as a {\it volume weight}.

\begin{exc}\label{weight exercise} 
(a) Show that $\v$ is a weight. 
\par (b) If $B$ and $B'$ are balls in $G$ and $\v$ and $\v'$ are the associated volume weights, then 
$\v$ and $\v'$ are comparable, i.e. there exists a constant $C>0$ such that 
$$ {1\over C} \v(z)\leq \v'(z) \leq C \v(z) \qquad (z\in Z)\, .$$ 
\end{exc} 

Given a homogeneous space $Z=G/H$ the volume weight $\v$ is closely tied to the 
harmonic analysis  of $Z$. The next lemma gives a first flavor how  geometry and analysis are linked.

\begin{lemma}\label{S-lemma}{\rm (Bernstein's invariant Sobolev-lemma,  \cite{B})} Let $1\leq p< \infty$ and  $k>{\dim G\over p}$. 
Fix a ball $B\subset G$ with associated volume weight $\v$. Then there exists a constant $C_B>0$ such that 
\begin{equation}\label{inv Sob}
|f(z)|\leq C_B \v(z)^{-{1\over p}} \|f\|_{p,k,Bz}
\end{equation}
for all $z\in Z$ and all
smooth functions $f$ on $Z$.   Here $\|\cdot\|_{p, k, Bz}$ refers to a $k$-th Sobolev norm 
of the $L^p$-norm $\|\cdot\|_p$ of $L^p(Z)$ restricted to $Bz$. 
\end{lemma}

\begin{proof} For fixed $z$ the estimate (\ref{inv Sob})
reduces to the classical Sobolev lemma, which is valid
for $k>\frac{\dim Z}p$ on any smooth manifold $Z$,
since by means of local coordinates 
it can be reduced to $\R^n$.

We first consider the special case $Z=G$, 
say with the left action of $G$. 
For simplicity let us assume that $G$ is unimodular. Then
the right invariance of Haar measure implies that
$\v$ is a constant function. Furthermore, in this case 
the right action allows us to immediately deduce 
the global estimate from the local estimate at $z=e$. 

Next we consider the general case of a unimodular 
homogeneous space $G/H$. Let $F\in C^\infty(G)$ 
be the pull-back of $f$ to $G$, then it follows from
what was said above that
$$|F(x)|\le C_B \|F\|_{p,k,Bx}\qquad (x\in G).$$
Here the Sobolev norm is given by an $L^p$-integral
of the derivatives of $F$ over $Bx\subset G$, and hence 
in order to relate to the corresponding integral for
$f$ it suffices to show that 
$$\int_B \phi(g\cdot z) \,dg \le C \v(z)^{-1}
\int_{Bz} \phi(y)\, dy$$
for all $z\in Z$ and all measurable positive
functions $\phi$ on $Z$,
with a constant $C$ which is independent of $z$ and $\phi$.

In order to see this we assume, as we may, that $\phi$
is supported on $Bz$. For each $y\in Bz$ we write $z=b^{-1}y$
and deduce
$$\int_B \phi(gz)\,dg 
=\int_{Bb^{-1}} \phi(gy)\, dg\le \int_{B^2} \phi(gy)\, dg.$$
By averaging this estimate over $y\in Bz$ we obtain
$$\int_B \phi(gz)\,dg \le \frac 1{\v(z)}
\int_{Bz} \int_{B^2} \phi(gy)\, dg\, dy$$
and with Fubini and invariance of $dy$ we conclude
$$\int_B \phi(gz)\,dg \le \frac {\vol_G B^2}{\v(z)}
\int_{Bz} \phi(y)\, dy$$
as claimed (in the last step it was used 
that $\supp \phi\subset Bz$).
\end{proof}

In general one would like to determine the growth behavior of the volume weight $\v$.  For that one needs to understand
the large scale geometry of the space $Z$. In case $Z$ is a real spherical space we will see later that 
this is in fact possible. 

For homogeneous spaces there is the following general criterion.

\begin{theorem} \label{vol bound} \cite{KSS}. Let $G$ be a real reductive group and $H<G$ a connected subgroup of $G$
with real algebraic Lie algebra.  Then the following statements are equivalent: 

\begin{enumerate}  
\item The volume weight $\v$ is bounded from below, i.e. there exists a $c>0$ such that $\v(z)\geq c$ for all 
$z\in Z$. 
\item $\hf$ is reductive in $\gf$, i.e.  $\ad_\gf|_\hf$ is completely reducible. 
\end{enumerate}
\end{theorem}

\par By a {\it real reductive group} we understand a real Lie group $G$ with finitely many connected components
such that there is a Lie group morphism  $\iota: G \to \GL(n,\R)$ with: 
\begin{itemize} 
\item $\iota(G)$ is closed and stable under matrix-transposition, 
\item $\ker \iota$ is finite.
\end{itemize}

\begin{exc} Consider $G=\Sl(2,\R)$ and 
$$H=N=\left\{ \begin{pmatrix} 1 & x \\ 0 & 1\end{pmatrix}\mid x\in \R\right\}\, .$$
\begin{enumerate}
\item Show that there is a natural identification of  $Z=G/N$ with $\R^2\bs\{0\}$ under which the Haar measure
on $Z$ corresponds to the Lebesgue measure $dx \wedge dy $ on $\R^2\bs\{0\}$.
\item For $t>0$ and $a_t:=\begin{pmatrix} t & 0 \\ 0 & {1\over t}\end{pmatrix}$ show that 
$$\v(a_t \cdot z_0) \asymp  t^2\,.$$
Hint: Take a ball $B$ which is of the form $B=B_N^t  B_A B_N$ where $B_N$ is a ball in $N$ and $B_A$ is a ball in the 
diagonal matrices.
\end{enumerate}
\end{exc} 

\begin{exc}  Use the invariant Sobolev lemma and Theorem \ref{vol bound} to deduce the following vanishing result.
For $G$ real reductive and $\hf<\gf$ algebraic and reductive in $\gf$ one has 
\begin{equation} \label{VAI} L^p(Z)^\infty \subset C_0(Z)\, .\end{equation}
Remark: The converse is also true, i.e. (\ref{VAI}) implies $\hf$ is reductive in $\gf$, see \cite{KSS}. 
\end{exc}

\section{Harish-Chandra modules and their completions} 
Given a real reductive group $G$ and a choice of morphism
$\iota$ the prescription $K:=\iota^{-1}(\operatorname{O}(n,\R))$ defines 
a maximal compact subgroup of $G$. We recall that all maximal compact subgroups of $G$ are conjugate
(Cartan's theorem).  A choice of $K$ yields an involutive automorphism $\theta:\gf\to \gf$,  called 
Cartan-involution,   with fixed point algebra $\gf^\theta=\kf$. For example, for our above choice of 
$K=\iota^{-1}(\operatorname{O}(n,\R))$ we would 
have $\theta(X)=-X^T$ if we view $\gf$ as a subalgebra of $\gl(n,\R)$.  
In general we denote by $\sf\subset \gf$ the $-1$-eigenspace of $\theta$ and recall that the polar map 
$$ K \times \sf\to G, \ \ (k,X)\mapsto k\exp(X)$$
is a diffeomorphism. In particular, $K$ is a deformation retract of $G$ and thus contains all topological 
information of $G$.

\par  We denote by $\hat K$ the set of equivalence classes of irreducible representations of $K$. 
Given an irreducible representation $(\pi,E)$ we let $[\pi]:=[(\pi,E)]$ be its equivalence class. We adopt the common 
abuse of notation and  write $\pi$ instead of $[\pi]$.
Since $K$ is compact its irreducible representations are finite dimensional, and 
they are essentially parametrized 
by their highest weights (there are some issues here if $K$ is not connected).

\begin{theorem}\label{HC thm}  {\rm (Harish-Chandra)}  Let $(\pi, E)$ be a unitary irreducible representation of $G$. Let 
$$V:=E^{{\rm K-fin}}:=\{ v\in E\mid \Span_\C\{ \pi(K)v\}\ \text{\rm is finite dimensional}\}\, .$$
Then the following assertions hold: 
\begin{enumerate} 
\item $V$ is dense in $E$, 
\item $V\subset E^\infty$ and $V$ is $\gf$-stable, 
\item $V$ is an irreducible module for $\U(\gf)$, 
\item For all $\tau\in\hat K$ one has 
$$ \dim \Hom_K (\tau, V)\leq \dim \tau <\infty\, .$$
\end{enumerate}
\end{theorem}

In other words to an irreducible unitary $G$-representation $(\pi,E)$  
we associate a module $V$ for $\U(\gf)$ 
(and $K$) which is irreducible (and hence of countable dimension) and which is some sort of skeleton of $(\pi,E)$. 
Historically this result was the beginning of the algebraization of representations of real reductive groups.
Here are the suitable definitions. 

\begin{definition} A {\it $(\gf,K)$-module} $V$ is a vector space endowed with two actions: 
$$\U(\gf) \times V \to V, \ \ (u,v) \mapsto u\cdot v$$
$$K\times V \to V, \ \ (k,v)\mapsto k\cdot v$$
such that the following conditions are satisfied: 
\begin{itemize}
\item The $K$-action is algebraic, i.e. for all $v\in V$ the space $V_v:=\Span\{ K\cdot v\}$ is finite dimensional 
and the action of $K$ on $V_v$ is continuous. 
\item The two actions are compatible, i.e.
$$ k\cdot (u\cdot v)= (\Ad(k)u)\cdot(k\cdot v)\qquad (k\in K, u\in \U(\gf), v\in V)\, .$$ 
\item The derived action of $K$ coincides with the $\U(\gf)$-action when restricted to $\U(\kf)$: 
$${d\over dt}\Big|_{t=0}\exp(tX) \cdot v = X\cdot v \qquad (X\in \kf)\, .$$

\end{itemize}
\end{definition} 

\begin{ex} For every Banach representation $(\pi, E)$ of $G$, the space 
$$V:=\{ v\in E^\infty\mid v \ \text{is $K$-finite}\}$$
is a $(\gf,K)$-module. Observe that smoothness of $K$-finite vectors is not automatic if 
the module is not irreducible. 
\end{ex}

\begin{definition} A $(\gf, K)$-module $V$ is called a {\it Harish-Chandra}-module provided that 
\begin{enumerate} 
\item  $V$ is {\it admissible}, that is, for all $\tau\in \hat K$ one has 
\begin{equation} \dim \Hom_K(\tau, V)<\infty\, .\end{equation}
\item $V$ has finite length as a $\U(\gf)$-module, that is, there is a sequence of $(\gf,K)$-submodules
$$ V_0=\{0\} \subset V_1 \subset V_2 \subset\ldots\subset V_m=V$$
such that $V_{i+1}/V_i$ is an irreducible $\U(\gf)$-module for all $0\leq i \leq m-1$. 
\end{enumerate} 
\end{definition} 

Having this notion, Theorem \ref{HC thm}  asserts in particular, that the $K$-finite vectors of an irreducible unitary 
representation form a Harish-Chandra module. 

\begin{rmk} (Rough classification of irreducible Harish-Chandra modules)   Let us denote by $\Zc(\gf)$ the center 
of $\U(\gf)$. Recall that the Harish-Chandra isomorphism describes $\Zc(\gf)$ as $S(\cf)^W$ where $\cf<\gf_\C$ is 
a Cartan subalgebra and $W$ the associated Weyl-group. 
\par We say that a $(\gf, K)$-module $V$ {\it admits an infinitesimal character}, provided that there exists a character 
$\chi: \Zc(\gf) \to \C$ such that $z\cdot v =\chi(z) v$ for all $z\in \Zc(\gf)$.  
It follows from Dixmier's version of Schur's lemma (see \cite{W1}, Lemma 0.5.1) that every irreducible
$\U(\gf)$-module $V$ admits an infinitesimal character $\chi_V$. 
Write now $\mathcal{HC}_{\rm irr}$ for the isomorphism 
classes of irreducible Harish-Chandra modules.  Then another theorem of Harish-Chandra asserts that 
the map 

$$ \mathcal{HC}_{\rm irr} \to \widehat{\Zc(\gf)}, \ \ V\mapsto \chi_V$$
has finite fibers. 
\end{rmk}

\begin{definition} By a globalization of a Harish-Chandra module $V$ we understand a $G$-representation 
$(\pi,E)$ modelled on a complete TVS $E$, such that $E^{{\rm K-fin}}\simeq_{(\gf,K)} V$. 
\end{definition} 

Globalizations of Harish-Chandra modules always exist. This is a consequence of the so-called 
Casselman subrepresentation theorem which we review next. 

\subsection{The subrepresentation theorem}

By a (real) parabolic subalgebra $\pf<\gf$ we understand a subalgebra such that $\pf_\C$ is a
parabolic subalgebra of $\gf_\C$.  A parabolic subgroup $P$ of $G$ is the normalizer of a parabolic subalgebra
$\pf$. 

\par Our concern here is with minimal parabolic subalgebras of $\gf$ and we recall their standard 
construction.   Recall the Cartan involution $\theta$ of $\gf$ and its eigenspace decomposition 
$\gf=\kf+\sf$. Let $\af\subset\sf$ be a maximal abelian subspace.  In this context we record 
the Lie-theoretic generalization of the spectral theorem for real symmetric matrices: all maximal 
abelian subspaces of $\sf$ are conjugate under $K$.
Fix such a maximal abelian subspace $\af\subset \sf$ and let $\mf:=\zf_\kf(\af)$ be the centralizer of 
$\af$ in $\kf$.  Note that $\zf_\gf(\af)= \af \oplus \mf$. Simultaneous diagonalization of 
$\ad_\gf|_\af$ then yields the root space decomposition 
$$ \gf=\af \oplus \mf \oplus \bigoplus_{\alpha\in\Sigma} \gf^\alpha$$
where $\Sigma\subset \af^*\bs\{0\}$ and 
$$\gf^\alpha=\{ Y \in \gf \mid (\forall X \in \af) \ [X,Y]=\alpha(X) Y\}\neq \{0\}\, .$$
Observe that all root spaces $\gf^\alpha$ are modules for $\af\oplus\mf$ which are typically 
not one-dimensional.  However, $\Sigma$ is a restricted root system and we let $\Sigma^+$ be a choice of positive 
system.  Associated to $\Sigma^+$ is the nilpotent subalgebra $\nf:=\bigoplus_{\alpha\in\Sigma^+}\gf^\alpha$.
\par On the group level we define subgroups of $G$ by  $A:=\exp(\af)$, $M:=Z_K(A)$ and $N:=\exp(\nf)$.  We recall 
the Iwasawa decomposition of $G$ which asserts that the map 
$$ K \times A\times N \to G, \ \ (k, a, n)\mapsto kan$$ 
is a diffeomorphism. It is of course not a homomorphism of groups,
but the restriction to $(M\times A)\ltimes N$ is, and the
image $P:=MAN$ defines a minimal parabolic subgroup 
of $G$, to which all other minimal parabolic subgroups are conjugate. The Iwasawa-decomposition implies that  $M=P\cap K$ and  $G=KP$. 

\begin{ex} (a) For $G=\GL(n,\R)$ one can take for $P$ the upper triangular matrices in $G$. 
Then $N$ are the upper triangular unipotent matrices,  $A=\diag(n,\R_{>0})$, and $M=\diag(n, \{ \pm 1\})$. 
\par (b)  Given $n\in \N$ we let $G=\SO_e(1,n)$ be the connected component of the invariance group 
of the quadratic form $q(x)=x_0^2 - x_1^2 - \ldots -x_n^2$ on $\R^{n+1}$.  Then $K=\SO(n,\R)$ 
(standard embedding in the lower right corner) is a maximal compact subgroup of $G$.  Further one can take 
\begin{eqnarray*}A&=&\left\{\begin{pmatrix} \cosh t  & & \sinh t \\ & \1 & \\ \sinh t  & & \cosh t \end{pmatrix}\mid t\in\R\right\} \simeq (\R_{>0}, \cdot)\, ,\\ 
M&=&\left\{\begin{pmatrix} 1 & & \\ & m & \\ & & 1 \end{pmatrix}\mid m\in \SO(n-1,\R)\right\} \simeq \SO(n-1,\R)\, ,\\ 
N&=&\left \{ \begin{pmatrix}1+ {1\over 2}\|v\|^2 & v & -{1\over 2}\|v\|^2\\ v^t & \1 & -v^t\\ {1\over 2}\|v\|^2 & v & 1-{1\over 2}\|v\|^2\end{pmatrix}
\mid v\in\R^{n-1}\right\} \simeq (\R^{n-1}, +)\, .
\end{eqnarray*}
\end{ex}

\bigskip

Next let $\sigma: P \to \GL(W)$ be a finite dimensional representation of $P$.  Attached to $\sigma$ is the 
smooth induced representation, defined by 
$$(\Ind_P^G \sigma)^\infty:=\{ f: G\to W\mid f\ \text{smooth}, f(gp)=\sigma(p)^{-1} f(g)\}\, . $$
Notice that there is a natural left action $L$ of $G$ on $(\Ind_P^G \sigma)^\infty$ given by 
$$L(g) f(x):=f(g^{-1}x) \qquad (g,x\in G, f\in (\Ind_P^G \sigma)^\infty)\, .$$
Next we topologize $(\Ind_P^G \sigma)^\infty$. Set 
$$(\Ind_M^K \sigma )^\infty :=\{ f: K\to W\mid f\ \text{smooth}, f(km)=\sigma(m)^{-1}f(k)\}\, .$$
Let $\|\cdot\|$ be any norm on the finite dimensional space $W$ and observe that 
$$p(f):=\sup_{k\in K} \|f(k)\|$$ 
defines a norm on $(\Ind_M^K \sigma )^\infty$ as $K$ is compact. The left regular representation 
of $K$ on $(\Ind_M^K \sigma )^\infty$ is continuous and a family of Sobolev norms $(p_k)_{k\in\N}$ then 
yields an $SF$-structure on the $K$-module $(\Ind_M^K \sigma )^\infty$.

Moreover, since $G=KP$ with $M=K\cap P$, the restriction map 
$$\Res: (\Ind_P^G \sigma)^\infty\to (\Ind_M^K \sigma )^\infty, \ \ f\mapsto f|_K$$
is a $K$-equivariant bijection. Via this map we transport the Fr\'echet structure of  
$(\Ind_M^K \sigma )^\infty$ to 
$(\Ind_P^G \sigma)^\infty$. 
It is not too difficult to show that: 

\begin{lemma}  $(L, (\Ind_P^G \sigma)^\infty)$ is an SF-representation of $G$. 
\end{lemma}

We write $\Ind_P^G \sigma$ for the space of
$K$-finite vectors of $(\Ind_P^G \sigma)^\infty$. 
 
 \begin{lemma} $\Ind_P^G \sigma$ is a Harish-Chandra module. \end{lemma}
 \begin{proof} (Sketch)  
 The map $\Res$ restricts to an isomorphism between 
$\Ind_P^G \sigma$  and the induced representation space $\Ind_M^K \sigma$ 
of $K$-finite functions in $(\Ind_M^K \sigma )^\infty$.
 It follows that $\Ind_P^G\sigma$ is admissible. 
 Next use a Jordan-H\"older series of $W$ to reduce to the fact that 
 $\sigma$ is irreducible, and in particular that $\sigma|_N$ is trivial.  Then it is not too hard too see that 
 $\Ind_P^G \sigma$ admits an infinitesimal character.   The assertion now follows from the general fact that 
 an admissible module which admits an infinitesimal character has finite length. 
 \end{proof} 

This brings us finally to the 

\begin{theorem} {\rm (Casselman subrepresentation theorem,
 \cite{cass})}  Let $V$ be a Harish-Chandra module. Then there
exists a finite dimensional representation $(\sigma, W)$ of $P$ and a $(\gf, K)$-embedding
$$ V \hookrightarrow \Ind_P^G \sigma\, .$$ 

\end{theorem}

\begin{rmk} (a) The subrepresentation theorem implies in particular that every Harish-Chandra module 
admits a globalization, even a Hilbert globalization.
In fact, if we complete $\Ind_P^G\sigma$ in $L^2(K\times_MW)$ we obtain 
a Hilbert globalization of $\Ind_P^G \sigma$.  The closure of $V$ in $L^2(K\times_MW)$ then yields a Hilbert globalization 
of $V$.  Other Banach globalizations are obtained by taking the closures in $L^p(K\times_M W)$. 
\par (b) Let $\nf$ be the Lie algebra of $N$, the unipotent radical of $P$. The subrepresentation theorem is then 
a consequence of a non-vanishing result for Harish-Chandra modules $V\neq\{0\}$: 
\begin{equation} V/\nf V\neq \{0\}\end{equation} 
combined with a variant of Frobenius reciprocity (see \cite{W1}, Th. 3.8.3 and Lemma 4.2.2). 
\end{rmk}
\subsection{The Casselman-Wallach globalization theorem}

Let $V$ be a Harish-Chandra module, $p$ a norm on the vector space $V$ and $V_p$ the Banach completion 
of $(V, p)$.  According to \cite{BK} a norm $p$ is called {\it $G$-continuous} provided the infinitesimal action of $\gf$ on $V$ exponentiates 
to a Banach representation of $G$ on $V_p$. 
With this terminology the globalization theorem can be phrased as follows: 

\begin{theorem}{\rm (Casselman-Wallach)} Let $V$ be a Harish-Chandra module and $p,q$ two $G$-continuous 
norms on $V$.  Then 
$$ V_p^\infty\simeq_{\rm SF-rep}V_q^\infty\, .$$
In particular, each Harish-Chandra module admits a unique SF-glo\-ba\-lization (up to isomorphism). 
\end{theorem} 

\begin{proof} See \cite{W2}, Ch. 11, or \cite{BK}. \end{proof}

\begin{rmk} An equivalent way to phrase the globalization theorem is as follows (see \cite{BK}). Let
$V=\bigoplus_{\tau\in\hat K}  V[\tau]$ be the $K$-isotypical decomposition of $V$. For every 
$\tau\in \hat K$, let $p_\tau:=p|_{V[\tau]}$ and $q_\tau:=q|_{V[\tau]}$. 
For every $\tau\in\hat K$ let $|\tau|$ be the Cartan-Killing norm of the highest weight which parametrizes 
$\tau$.  Then there exists a constant $N\in\N$ such that 
\begin{equation}  p_\tau \leq (1+|\tau|)^N q_\tau \qquad (\tau\in\hat K)\, ,\end{equation}
that is, $G$-continuous norms are polynomially comparable on $K$-types. 
\end{rmk}

\section{Generalized matrix coefficients}

\subsection{Matrix coefficients on groups} 
Let $G$ be a locally compact group and $(\pi, E_\pi)$ be a unitary representation on some Hilbert 
space $(E_\pi, \la\cdot, \cdot\ra)$. 
For all $v,w\in E_\pi$ we form the {\it matrix-coefficient} 
$$m_{v,w}^\pi(g):= \la \pi(g) v, w\ra \qquad (g\in G)\, .$$
Note that matrix-coefficients are continuous functions on $G$. 

When $G$ is compact the Peter-Weyl theorem asserts that the space of
matrix coefficients is uniformly dense in $C(G)$, and hence also in
$L^2(G)$. As a consequence one obtains that $L^2(G)$ is the Hilbert space
direct sum of its isotypic components $L^2(G)^\pi$, and that 
$(v,w)\mapsto m_{v,w}^\pi$ provides a $G\times G$ equivariant
isomorphism $E_\pi\otimes E_\pi\to L^2(G)^\pi$. 

In particular, the density in $C(G)$ implies that the matrix coefficients
separate points on $G$ (and vice versa, by the Stone-Weierstrass
theorem). This result holds more generally without
assuming compactness of $G$. Let $\hat G$ denote the set of equivalence classes
of unitary irreducible representations of $G$.

\begin{theorem} {\rm (Gelfand-Raikov)} Let $G$ be a locally compact group. Then the functions
$$ \{ m_{v,w}^\pi\mid \pi \in\hat G, v,w\in E_\pi\}$$ 
separate points on $G$. 
\end{theorem} 

Phrased loosely, the theorem says that the geometry of $G$ is determined by its matrix coefficients. 

\par Let us now be more restrictive on the group $G$. 

\begin{theorem}\label{H-M thm}{\rm(Howe-Moore)} Let $G$ be a  non-compact Lie group with simple Lie algebra 
and compact center. 
Let $\pi\in \hat G\bs \{\1\}$.  Then, for all $v,w\in E_\pi$: 
$$ m_{v,w}^\pi \in C_0(G)=\{ f\in C(G)\mid \lim_{g\to \infty} f(g) =0\}\, .$$
\end{theorem} 

This result has, at first glance, a quite frustrating corollary for "ordinary" matrix 
coefficients on homogeneous spaces 
$Z=G/H$. Recall from (\ref{PW1})
that for compact groups the Peter-Weyl theorem implies the
$G$-equivariant decomposition 
$$L^2(G/H)=\hat\bigoplus_{\pi\in\hat G} E_\pi\otimes E_\pi^H$$
so that $E_\pi^H$ controls the embeddings of $\pi$ in $L^2(Z)$
(this could also be derived from the Frobenius reciprocity theorem). However:

\begin{cor}\label{E^H=0} Let $G$ be as in Theorem \ref{H-M thm}
and let $H<G$ be a closed and non-compact subgroup of $G$. Then 
for all $\pi \in \hat G\bs\{\1\}$ one has: 
$$E_\pi^H:=\{ v\in E\mid\forall h\in H, \  \pi(h)v=v \} =\{0\}\, .$$
\end{cor}

\begin{proof} Exercise.\end{proof}

\subsection{Generalized matrix coefficients} 
Here we let $G$ be a Lie group and $H<G$ be a closed subgroup. We have already seen that 
for a Banach representation $(\pi, E)$ of $G$ the chances that $E^H\neq\{0\}$ are slim. 
Instead of looking at $E^H$ we look at a larger space. 
\par Let $E^{-\infty} $ be the strong dual of  the Fr\'echet space $E^\infty$.  We refer to $E^{-\infty}$ as the 
space of distribution vectors of the representation $(\pi,E)$.  Then for every $\eta\in (E^{-\infty})^H$ and 
$v\in E^\infty$ we form the {\it generalized matrix coefficient}
$$m_{v,\eta} (z) =\eta(\pi(g)^{-1}v) \qquad (z=gH, g\in G)$$
which is a smooth function on $Z$. 
Let us see in an example that 
in spite of Corollary \ref{E^H=0}
we can obtain a huge supply of functions in that way. 
\begin{ex} We assume that 
$Z$ is unimodular and consider the left regular representation $L$ of $G$ on $E=L^p(Z)$
for $1\le p<\infty$.
We have already seen that there is a continuous inclusion $\iota: E^\infty \hookrightarrow  C^\infty (Z)$. 
Since evaluation at the base point
$$\delta_{z_0}: C^\infty(Z) \to \C, \ \ f\mapsto f(z_0)$$ 
is a continuous $H$-invariant functional, the composition $\eta:=\delta_{z_0} \circ \iota$ defines an 
element of $(E^{-\infty})^H$.  Let now $f\in E^\infty$ be arbitrary.  

Then we get for $z=gH\in Z$ that 
$$m_{f,\eta}(z)=\eta(L(g)^{-1} f) = \eta (f(g\cdot )) = f(z)\, , $$
that is, we recover $f$ as  a generalized matrix coefficient. 
\end{ex} 
 
The argument in the preceding example can be formalized into a statement of Frobenius reciprocity.
Recall that for any Banach representation $(\pi,E)$
we equipped $E^\infty$ with a Fr\'echet topology. We can then consider the vector space
$\Hom_G^{\rm cont} (E^\infty, C^\infty(G/H))$  of continuous 
intertwiners $E^\infty\to C^\infty(G/H)$.

\begin{lemma}\label{Frob} 
{\rm (Smooth Frobenius reciprocity)}  Let $(\pi,E)$ be a Banach representation of a 
Lie group $G$
and let $H<G$ be a closed subgroup.
Then
\begin{equation}\label{Fr1} \Hom_G^{\rm cont} (E^\infty, C^\infty(G/H)) \simeq (E^{-\infty})^H\, .\end{equation} 
\end{lemma}
 \begin{proof} Let $S\in \Hom_G^{\rm cont} (E^\infty, C^\infty(G/H))$ and $\delta_{z_0}: C^\infty(G/H)\to \C$ 
 be the point evaluation at $z_0$ as before. As $\delta_{z_0}$ is continuous and $H$-invariant, we obtain 
 via $\eta:=\delta_{z_0}\circ S$ a continuous $H$-invariant functional on $E^\infty$, 
 that is $\eta\in (E^{-\infty})^H$.
 \par Conversely, let $\eta\in (E^{-\infty})^H$ be given.  Then 
 $$S_\eta:  E^\infty\to C^\infty(G/H), \ \ v\mapsto m_{v,\eta}$$
 is $G$-equivariant and continuous. The calculation in the example shows that 
 if $\eta=\delta_{z_0}\circ S$ then $S_\eta=S$.
 \end{proof}

 This result has a useful dual version. For a Banach 
representation $(\pi,E)$ we denote by $(\pi',E')$ the dual 
representation on the dual Banach space $E'$, and
recall that $E$ is called reflexive when the natural inclusion 
$E\to E''$ is an isomorphism.
 
\begin{lemma}\label{dual Frobenius} {\rm (Dual smooth Frobenius reciprocity)}  
Let $Z=G/H$ and let $(\pi,E)$ be a reflexive
Banach representation. 
For each $\eta\in (E^{-\infty})^H$ and $f\in C_c^\infty(Z)$ the linear form
 \begin{equation}\label{F-t} 
v \mapsto \int_Z m_{v,\eta}(z) f(z)\, dz
 \end{equation} 
on $E^\infty$ extends continuously to $E$ and defines an element $T_\eta( f)\in (E')^\infty$.
The map $\eta \mapsto T_\eta$ then provides a linear isomorphism
 \begin{equation}\label{Fr2} 
  (E^{-\infty})^H \simeq \Hom_G^{\rm cont}(C_c^\infty(G/H), (E')^\infty) \,.
 \end{equation} 
\end{lemma}

The extension $T_\eta f$ of (\ref{F-t}) is the {\it Fourier transform} 
of $f$ at $(\pi,\eta)$.

\begin{proof}
The proof is essentially just by dualizing, but as it 
involves some non-trivial
functional analysis we provide the details. 

Recall that the strong dual of
  $C^\infty(Z)$ equals $\D_c(Z)$, the space of compactly supported distributions. 
  We let $\M_c^\infty(Z)\subset \D_c(Z)$ denote
 the subspace of compactly supported smooth measures, and note that $\M_c^\infty(Z)\simeq C_c^\infty(Z)$
 via a non-canonical isomorphism; for example if $Z$ is unimodular and $dz$ is the invariant 
 measure, then such an isomorphism would be given by 
  $$ C_c^\infty(Z) \to \M_c^\infty(Z),\ \ f\mapsto f\cdot dz, $$ 
 but of course any positive measure in the Lebesgue measure class will do. 
 
 Let $\eta \in (E^{-\infty})^H$ and define 
 $S=S_\eta\in \Hom_G^{\rm cont} (E^\infty, C^\infty(Z))$
 as in Lemma \ref{Frob} by $S(v)=m_{v,\eta}$.
 The dual morphism $S': \D_c(Z) \to E^{-\infty}$ is then continuous and $G$-equivariant.  
We claim that $ S'(\M_c^\infty(Z))\subset (E')^\infty$ and that 
\begin{equation}\label{S' image}
S'\big|_{\M_c^\infty(Z)}: \M_c^\infty(Z) \to (E')^\infty
\end{equation}
is continuous. 
\par To see that, let $E^{-m}$ be the strong dual of 
 $(E, p_m)$ with $m\in \N_0$
(see (\ref{defi Sobolev norm})).  
Note that $E^{-\infty}$ equals the increasing union
$E^{-\infty} =\bigcup_{m\in \N} E^{-m}$, and that each 
 $E^{-m}$ is a $G$-invariant Banach subspace of $E^{-\infty}$ which is continuously included into $E^{-\infty}$. 
 On the other hand fix a 
 compact set $C\subset Z$ and denote by $\M_C(Z)$ the space of finite Borel measures 
 with support in $C$.  Note that $\M_C(Z)$ is a Banach space which is 
 continuously included in $\D_c(Z)$.  Therefore we obtain with 
 $$S'\big|_{\M_C(Z)}: \M_C(Z) \to E^{-\infty}=\bigcup_{n\in\N} E^{-n}$$
 a continuous linear map and can apply the Grothendieck factorization theorem (cf.~\cite{Gr}, Ch.~4. Sect.~5., Th.~1): 
 we obtain an  $m$ such that $S'(\M_C(Z))\subset E^{-m}$, slightly more precisely, the map 
$S'\big|_{\M_C(Z)}: \M_C(Z) \to E^{-m }$ is continuous. 
Now note that $\M_c(Z)$ can be obtained as the linear span 
 $$\M_c(Z) =\Span_\C\{ L(g)  \M_C(Z)\mid g\in G\}$$ 
by covering the support of any given compactly supported measure
with finitely many $G$-translates of the interior of $C$,
and then applying an associated partition of unity.
Hence the fact that $E^{-m}$ is $G$-invariant
implies that we obtain a continuous $G$-equivariant linear map: 

$$ T_{\M}: =S'\big|_{\M_c(Z)} : \M_c(Z)   \to E^{-m}\, .$$ 
On the other hand, taking smooth vectors of the $G$-modules 
$\M_c(Z)$ and $E^{-m}$ yields that  
$(\M_c(Z))^\infty =\M_c^\infty(Z)$ and 
$(E^{-m})^\infty= (E')^\infty$. 
The claim (\ref{S' image}) follows with $T_\eta:=T_{\M}^\infty$ and the continuous linear $G$-morphism

$$T_\eta: \M_c^\infty(Z) \to (E')^\infty, \ \  \mu \mapsto \int_Z  \pi(g) \eta \ d\mu(gH)\, .$$

In particular, we obtain the map $T_\eta$ as a
continuous extension of (\ref{F-t}).
It is easily seen that $\eta\mapsto T_\eta$ is injective.

\par Conversely, let  $T: \M_c^\infty(Z) \to (E')^\infty$ be a $G$-morphism. This $G$-morphism
yields a $G$-morphism $T_0: \M_c^\infty(Z) \to E'$. By the definition of the topology 
of $\M_c^\infty(Z)$ we obtain for every compact set $C\subset Z$  
a $k=k(C)\in\N$ such that $T_0\big|_ {\M_C^\infty(Z)}$ extends 
to a continuous linear map $T_{C,k}: \M_C^k(Z) \to E'$.  As $T_0$ was $G$-equivariant we obtain with 
 partition of unity argument as above a $G$-morphism 
$T_k: \M_c^k(Z) \to  E'$. 
Since $E$ is reflexive we obtain a $G$-morphism 
$$(T_k)': E \to \D^{k}(Z)$$
by dualizing.
Taking smooth vectors finally provides a continuous homomorphism
$$S: E^\infty \to \D^{k}(Z)^\infty=C^\infty(Z)$$
which by Lemma \ref{Frob} equals $S_\eta$ for
some $\eta\in (E^{-\infty})^H$. 
It follows from this construction that
$T=T_\eta$,
and thus the proof is complete.
\end{proof}

\section{Real spherical spaces}

\par We assume now that $G$ is a connected reductive real algebraic group, that is, there exists a complex reductive 
group $G_\C$ attached to $\gf_\C:=\gf\otimes_\R\C$ such that $G$ is the analytic subgroup of $G_\C$ which 
is associated to the subalgebra  $\gf<\gf_\C$.  

\par Further we let $H<G$ be an algebraic subgroup, that is, there is a complex algebraic subgroup $H_\C<G_\C$ 
such that $H_\C\cap G=H$.  This way we obtain a $G$-equivariant embedding 
  \begin{align*}
    Z=G/H &\hookrightarrow Z_\C=G_\C/H_\C      \\
    gH &\mapsto gH_\C\, . 
  \end{align*}

\par Let now $P<G$ be a minimal parabolic subgroup.  We say that $Z$ is real spherical provided that 
the $P$-action on $Z$ admits open orbits. 
Equivalently $Z$ is real spherical if there exists a minimal parabolic subalgebra $\pf<\gf$ such that 
\begin{equation} \label{RS}\gf=\pf+\hf   \end{equation} 
where the sum is not necessarily direct.  If (\ref{RS}) is satisfied, then we refer to $(\gf, \hf)$ as a real spherical pair. 

\begin{rmk} In case both $\gf$ and $\hf$ are complex, then a real spherical pair $(\gf,\hf)$ is simply called spherical. 
Spherical pairs with $\hf$ reductive where classified by Kr\"amer (\cite{Kr} for $\gf$ simple) and by
Brion and Mikityuk 
(\cite{Brion}, \cite{Mik} for $\gf$ semi-simple). It is easy to see that a real pair
 $(\gf,\hf)$ is real spherical if its complexification  $(\gf_\C,\hf_\C)$ is spherical,
 such pairs are said to be {\it absolutely spherical}. The converse is not true,
 there exist many real spherical pairs which are not absolutely spherical.
 
If $\nf\triangleleft \pf$ is the nilradical of $\pf$, then a real spherical subalgebra $\hf$ must satisfy the 
 {\it dimension bound}
$$\dim \hf\geq \dim\nf\, .$$ 
The dimension bound is a decisive tool in the classification of (complex) spherical 
pairs, but due to the presence of large Levi-factors in real semi-simple Lie algebras the dimension bound can become 
rather weak for real pairs. For example if $\gf=\so(1,n)$, then $\dim\nf= n-1$ is very small.  For that reason it might be 
a bit surprising that it is still possible to classify
reductive real spherical subagebras 
(see \cite{KKPS}, \cite{KKPS2}). 
\end{rmk}
 
 \subsection{Local structure theorem}
If $L$ is a real reductive group, then we denote by $L_{\rm n}\triangleleft L$ the connected 
normal subgroup which corresponds to the sum of all non-compact factors of $\lf$.

\begin{theorem}{\rm (Local structure theorem \cite{KKS})}  Let $Z=G/H$ be a real spherical space and $P<G$ be 
a minimal parabolic subgroup such that $PH$ is open. Then there exists a unique  parabolic subgroup 
$Q\supset P$ with Levi-decomposition $Q=L\ltimes U$ such that: 
\begin{enumerate}
\item $PH=QH$, 
\item $Q\cap H=L\cap H$, 
\item $L_{\rm n}\subset L\cap H$, 
\item $QH/H \simeq U \times L/L\cap H$. 
\end{enumerate}
\end{theorem}
 
 A parabolic $Q$ as above will be called {\it $Z$-adapted}.  
 
\subsection{Spherical representation theory}

Let $V$ be a Harish-Chandra module and $V^\infty$ its unique SF-completion. 
We say that $V$ is {\it spherical} provided that $(V^{-\infty})^H\neq 0$.  If $0\neq \eta\in (V^{-\infty})^H$, then we 
refer to $(V,\eta)$ as a {\it spherical pair}. 

The Levi-decomposition $Q=L\ltimes U$ defines us in particular an opposed parabolic subgroup 
$\oline Q = L \oline U$ to $Q$.  Further we let $G=KAN$ be an Iwasawa decomposition such that 
$A\subset L $ and $N\subset Q$.

\begin{theorem} \label{KS-thm} {\rm \cite{KS}} The following assertions hold for a real spherical space: 
\begin{enumerate}
\item {\rm(Spherical subrepresentation theorem)} Every irreducible spherical Harish-Chandra module $V$ admits a
$(\gf,K)$-embedding 
$$V\hookrightarrow \Ind_{\oline Q}^G \sigma$$
where $\sigma$ is an irreducible finite dimensional representation of $\oline Q$ which is trivial on $ \oline U$
and $\Hom_{\lf\cap \hf, L\cap H\cap K} (\sigma,\C)\neq \{0\}$. 

\item {\rm(Dimension bound)}   Let $\oline \qf_1:=\lf\cap\hf +\oline\uf$. Then 
$$\dim (V^{-\infty})^H \leq  \dim \big(V/ \oline\qf_1 V\big)^{L\cap H\cap K}<\infty $$
\end{enumerate} 
\end{theorem}

\begin{rmk} \label{rmk subrep}
(a) The subrepresentation theorem for symmetric spaces 
(all of which are real spherical) was proved by 
Delorme in \cite{Patrick}.   We would like to point out that the proof of (1) in 
 \cite{KS}  does not rely on the subrepresentation theorems of 
Casselman and Delorme. 
\par (b) Since $\oline\nf\subset\oline\qf_1$
the second inequality in (2) 
follows immediately from the 
Casselman-Osborne lemma (\cite{W1}, 
Prop.~3.7.1 with Cor.~3.7.2) according to which 
the Jacquet module
$J(V):=V/\oline{\nf}V$
is finite dimensional. The first
inequality in (2) is generically sharp.  A non-effective bound 
with a different proof was obtained previously in \cite{KO}. 
For symmetric spaces, the finite dimensionality
was established in \cite{vdB}.
\end{rmk}

\begin{proof} We sketch a proof of Theorem \ref{KS-thm} (2) for 
$G=\Sl(2,\R)$, based on the same
ideas which enter the general proof.  
For simplicity we assume
also that $\hf$ is one-dimensional. In this case $\qf=\pf$ and
the bound amounts just to
$$\dim (V^{-\infty})^H \leq  \dim V/ \oline\nf V .$$
Let 
$$  E:=\begin{pmatrix} 0 & 1 \\ 0 & 0\end{pmatrix}, \quad  F:=\begin{pmatrix} 0 & 0 \\ 1 & 0\end{pmatrix}, 
\quad Y:=\begin{pmatrix} 1& 0 \\ 0 & -1\end{pmatrix}$$
and set 
$$ \nf=\R E, \quad \ \oline\nf = \R F, \quad \af = \R Y\, .$$
Then $\pf=\af +\nf$ is a minimal parabolic subalgebra of $\gf$.
We assume $\gf=\hf + \pf$.
Hence $\hf=\R U$ where $U= F  + cY + dE$ for some $c, d\in \R$. 

Let  $\af^- \subset\gf$ denote the closed half line $\R_{\leq 0} Y$,
and consider its open neighborhoods
$\af_\e^-:=  (-\infty,\e) Y $ for each $\e>0$. 
For $t\in \R$ set 
$a_t:=\exp(tY)$. 
Now if we choose the maximal compact subalgebra to be $\kf=\R(E-F)$ then
\begin{equation}\label{kah}
\gf = \kf   + \af + \Ad(a_t)\hf
\end{equation}
except if $e^{-4t}=-d$. It will be convenient for some $\epsilon>0$
to have
(\ref{kah}) for all $tY\in\af_\epsilon^-$, 
and this we accomplish by choosing instead
$\kf=\R\Ad(a_s)(E-F)$ for some sufficiently large $s\in\R$.

\par  Now comes a piece of information which we use as a black box, the existence of convergent asymptotic  expansions
of generalized matrix coefficients on the compression cone\footnote{See Subsection \ref{subsection polar} for the definition of the compression cone.}  (see \cite{KS}, Sect.~5 and \cite{KKS2} Sect.~6).  
When (\ref{kah}) is valid as above one can 
show that
there is $0<\e'<\e$, a finite set of leading exponents $\E= \E (V)\subset \C$, and a number $N= N(V)\in \N$ such that for 
all $v\in V$ one has  absolutely convergent expansions: 
\begin{equation}\label{power} 
 m_{v,\eta}(a_t)= \sum_{\lambda\in \E} \sum_{n\in \N_0} 
 \sum_{0\leq k\leq N}  e^{t (\lambda + 2n)} t^k  c_{k,n,\lambda}
\quad (-\infty<t<\e')\end{equation}
with coefficients $c_{k,n,\lambda}\in \C$, depending linearly on $v$. 
\par To the Harish-Chandra module $V$  we associate the
finite-dimensional space 
$J(V)=V/ \oline \nf V$  (see Remark 
\ref{rmk subrep}).  Since $\oline \nf$ is an ideal of $\oline \pf = \af + \oline \nf$, we see that $J(V)$ is a
module for $\oline \pf$. 
The $\af$-spectrum of $J(V)$ is closely tied to the set of exponents, $\E(V)$, introduced above
(see \cite{KKS2}, Sect. 7 and the reasoning below). 
On the side we remark here 
that $J(V)$ is two-dimensional for a generic irreducible Harish-Chandra module $V$.  
In order to avoid heavy 
notation involving matrix exponential functions with Jordan blocks  we assume here 
that $J(V)$ is one-dimensional.   Let $v_0\in V$ be such that $\oline v_0:= v_0+ \oline\nf V $ generates $J(V)$. As we 
assume that $J(V)$ is one-dimensional we have $Y \oline v_0 = -\lambda \oline v_0$ for some $\lambda\in \C$.   

\par Since $J(V)$ is one-dimensional we have to show that 
$(V^{-\infty})^H$ is at most one-dimensional. 
For $\eta\in (V^{-\infty})^H$  and $v\in V$ we set 
$$f_v(t) =  m_{v,\eta}(a_t):= \eta( a_t^{-1} v )$$
with the simplified notation $f:=f_{v}$ when $v=v_0$. 
The one-dimension\-ality of $(V^{-\infty})^H$
will be established by showing that the number
\begin{equation}\label{limit}
c_{-\infty}=\lim_{t\to -\infty} e^{-\lambda t}f(t)
\end{equation}
exists and determines $\eta$ uniquely.
 
Let further $ w\in V$ be such that $ Y v_0 =- \lambda v_0 + F w$ and set $u:=Fw$.  Let $r(t):=-f_u(t)$. 
We observe that $f$ satisfies the first order differential equation
$$ f'(t) = \lambda f (t)  + r(t)  \qquad (t\in \R)$$
which features the general solution 
\begin{equation} \label{sol} f(t)= e^{\lambda t}\Big(  f(0) +\int_0^t e^{-s\lambda} r(s) \ ds \Big)\, .\end{equation}
 
\par We note that there 
is an a priori bound 
\begin{equation} \label{bound} |f_v(t)|  \leq C_v e^{\Lambda t }  \qquad (t\leq 0) \end{equation} 
for some $\Lambda \in \R$   and $C_v>0$. This is a simple consequence of the fact that $\eta$ is a distribution vector (see
\cite{KSS2}, (3.2) and (3.7)), but it also 
follows from the asymptotic expansion (\ref{power}). 
The key observation is now that for $u\in \oline\nf V$ one then has an improved bound 
\begin{equation} \label{improved bound} |f_u (t)|\leq C_u'  e^{(\Lambda +2)t} \qquad (t\leq 0, u \in \oline \nf V)\, .\end{equation}
In fact write $u = Fw$ and note that 
\begin{eqnarray*}f_u(t) &=&  \eta(a_t^{-1} Fw)  =  e^{2t} \eta(F a_t^{-1} w)= e^{2t} \eta( (U - cY - dE)a_t^{-1} w)\\ 
& =& -e^{2t} \eta ((cY + dE)a_t^{-1}w)  = -ce^{2t} f_{Yw}(t) - d e^{4t} f_{Ew}(t)\end{eqnarray*} 
and thus (\ref{improved bound}) follows from (\ref{bound}).  

\par It is now a matter of simple combination of 
(\ref{sol}),  (\ref{bound}) and (\ref{improved bound})  to deduce that (\ref{bound}) must hold with
$\Lambda=\re \lambda$. 
Then we also have (\ref{improved bound})
with that value of $\Lambda$, and the existence of
(\ref{limit}) follows from (\ref{sol}). 
Moreover we can write
\begin{equation} \label{sol inf} 
f(t)= e^{\lambda t}\Big(  c_{-\infty} + \int_{-\infty}^t  e^{-s\lambda} r(s) \ ds \Big)\, .
\end{equation}

Observe that $c_{-\infty}$ depends linearly on $\eta$.  
Finally if $c_{-\infty}=0$
then (\ref{sol inf}) implies that $|f(t)|\leq C e^{(\re \lambda +2)t}$, which together
with (\ref{improved bound}) allows us to replace $\Lambda$ by $\Lambda +2$. Iterating we can achieve 
$\Lambda$ as large as we wish and thus 
conclude $|f(t)| \ll  e^{t N}$ for all $t\leq 0$ and $N>0$.  Hence 
it follows from (\ref{power})  that $f\equiv 0$.  Since we may assume that $v_0$ belongs to a $K$-type and $K AH$ has open interior on $G$ we get that $m_{v_0,\eta}\equiv 0$
and thus $\eta =0$. This proves that $\eta\mapsto c_{-\infty}$ is injective.
\end{proof}

\begin{problem} \label{c-conj}{\rm (Comparison conjectures)}
Let $V$ be a Harish-Chandra module and $V^\infty$ its unique SF-completion. Further let 
$\hf<\gf$ be a real spherical subalgebra.  Then for a generic choice of $K$ all homology 
groups $H_p(V,\hf)$ are finite dimensional, see \cite{AGKL}, Prop. 4.2.2. 
The comparison conjecture of \cite{AGKL}  (for $\hf=\nf$ due to Casselman) asserts for all $p\geq 0$: 
\begin{enumerate}
\item $H_p(V^\infty, \hf)$ is separated, 
\item $H_p(V, \hf)\simeq H_p(V^\infty,\hf)$. 
\end{enumerate}
Notice that (1) for $p=0$ means that $\hf V^\infty$ is closed in $V^\infty$, in other words the first order 
partial differential operator 
$$\hf \otimes V^\infty \to V^\infty, \ \ X\otimes v \mapsto X\cdot v$$
has closed range.  This special case was established in \cite{AGKL} in case $\hf$ is absolutely spherical, that is 
$\hf_\C$ is spherical in $\gf_\C$. An important application is presented in \cite{KKS2}, Section 7. 
\end{problem}

\section{Disintegration of group representations} 

\subsection{Direct integrals of Hilbert spaces and the theorem of Gelfand-Kostyuchenko}

Let $\Lambda$ be a topological space with countable basis and let $\mu$ be a $\sigma$-finite 
Borel measure on $\Lambda$. 

\par Suppose that for each $\lambda\in\Lambda$ we are given a Hilbert space $\Hc_\lambda$ with inner product 
$\la\cdot, \cdot\ra_\lambda$. We identify $\prod_{\lambda}\Hc_\lambda$ with 
functions $s: \Lambda \to \coprod_{\lambda\in\Lambda}\Hc_\lambda$ which satisfy $s(\lambda)\in\Hc_\lambda$.  
We refer to the elements $s\in \prod_{\lambda}\Hc_\lambda$ as sections. 

By a {\it measurable family of Hilbert spaces over $\Lambda$} we understand a subspace of sections
$\F \subset \prod_{\lambda}\Hc_\lambda$ which satisfies the following axioms:
\begin{itemize} 
\item For all $s, t\in \F$, the map $\lambda \mapsto\la s(\lambda), t(\lambda)\ra_\lambda$ is measurable.
\item If $t\in \prod_{\lambda}\Hc_\lambda$ is a section such that $\lambda\mapsto \la t(\lambda), s(\lambda)\ra_\lambda$
is measurable for all $s\in\F$, then $t\in\F$.
\item There exists a countable subset $(s_n)_{n\in\N}\subset \F$ such that  for each $\lambda\in\Lambda$
the values
$\{s_n(\lambda)\mid n\in\N\}$ span a
dense subspace in $\Hc_\lambda$. 
\end{itemize}

Given a measurable family of Hilbert spaces over $\Lambda$ and $s\in\F$, we call $s$ {\it square integrable}
provided that 
$$\int_\Lambda \la s(\lambda), s(\lambda)\ra \ d\mu(\lambda) <\infty\, .$$
The space of square integrable sections is denoted by 
$$\int_\Lambda ^\oplus \Hc_\lambda \ d\mu(\lambda)$$
and referred to as the direct integral of the measurable family $(\Hc_\lambda)_\lambda$.  Given our assumptions
on $\Lambda$, the direct integral is a separable Hilbert space as well.   

The basic example is where all $\Hc_\lambda=\C$ and then the direct integral is simply 
$L^2(\Lambda, \mu)$.

\par Let $E$ be a topological vector space and $T: E\to \int_\Lambda^\oplus \Hc_\lambda \ d\mu(\lambda)$
be a continuous linear map. We say that $T$ is {\it pointwise defined}, provided that there exists a set of measure 
zero $\Lambda_0\subset \Lambda$ and continuous linear maps 
$$\operatorname{ev_\lambda}: E \to \Hc_\lambda  \qquad (\lambda\in \Lambda\bs \Lambda_0)$$
such that 
$$T(v) (\lambda) = \operatorname{ev_\lambda}(v) \qquad (\lambda\in \Lambda\bs \Lambda_0, v\in E)\, .$$

\begin{ex} The identity $L^2(\R)\to L^2(\R)$ is not pointwise defined. The inclusion $\Sc(\R) \to L^2(\R)$ is 
pointwise defined. 
\end{ex} 
	
\par Recall that a continuous linear map $T: \Hc_1\to \Hc_2$ between Hilbert spaces is called Hilbert-Schmidt, 
provided that there exists an orthonormal basis $(v_n)_{n\in \N}$ of $\Hc_1$ such that $\sum_{n\in\N} \|T(v_n)\|^2<\infty$.  

\begin{theorem}\label{G-K}{\rm (Gelfand-Kostyuchenko)}  
Let $T: \Hc \to \int_{\Lambda}^\oplus \Hc_\lambda \ d\mu(\lambda)$ 
be a Hilbert-Schmidt operator. Then $T$ is pointwise defined. 
\end{theorem}
\begin{proof} \cite{B}, Th. 1.5. Here is a sketch. Let 
$(v_n)_{n \in\N}$ be an ONB of $\Hc$.  Then, since $\N$ is countable, 
there exists a set of measure zero $\Lambda_1\subset \Lambda$ such that $T(v_n)(\lambda)=:v_n(\lambda)$ is defined 
for all $\lambda\in \Lambda\bs\Lambda_1$ and $n\in\N$.  Then, 
\begin{eqnarray*}\infty>\sum_{n\in \N} \| T(v_n)\|^2 &=& \sum_{n\in\N}\int_{\Lambda\bs\Lambda_1}  \|v_n(\lambda)\|_\lambda^2 \ d\mu(\lambda)\\
&=&\int_{\Lambda\bs\Lambda_1}  \sum_{n\in \N}  \|v_n(\lambda)\|_\lambda^2 \ d\mu(\lambda)\end{eqnarray*}
We conclude that there exists a set of measure zero $\Lambda_1\subset \Lambda_0\subset \Lambda$ such that 
$$ \sum_{n\in \N} \|v_n(\lambda)\|_\lambda^2<\infty \qquad (\lambda\in \Lambda\bs \Lambda_0)\, .$$
Via Cauchy-Schwartz this allows us to define 
$$ \qquad\qquad\quad\,\,\operatorname{ev}_\lambda(v):=  \sum_{n\in \N} \la v,v_n\ra  v_n(\lambda)  
\qquad (\lambda\in \Lambda\bs\Lambda_0, v\in\Hc)\, .\qquad\qedhere$$
\end{proof}

\begin{exc} \label{weighted Sobolev}Let $M$ be a manifold and let $\mu$ 
be a measure on $M$ which is locally comparable
to the Euclidean measure, i.e. for all $m\in M$ there exists a neighborhood $U$ of $M$, a diffeomorphism 
$\phi: U \to V$ with $V\subset \R^n$ open such that $\phi_*\mu = f \ dx $ for a measurable  function 
$f:V\to \R$ which is bounded from above and below by positive constants. 
Let $L^2(M)=\int_{\Lambda}^\oplus \Hc_\lambda \ d\mu(\lambda)$. Show that the inclusion 
$C_c^\infty(M)\to L^2(M)$ is pointwise defined.  [Hint: Consider first the case where $M=(-1,1)^n\subset\R^n$
and use the fact that a Sobolev space $\Lc$ of high enough order on $M$ gives a  Hilbert-Schmidt embedding 
$\Lc \to L^2(M)$.]
\end{exc}

\subsection{Disintegration} 
Let $G$ be a type  I-Lie group (see \cite{W2} p. 333). 
The precise definition is not important for us; it is sufficient to know that 
real  reductive Lie groups and nilpotent Lie groups are all type I.

\par The disintegration theorem for unitary representations $(\pi,\Hc)$ of a type  I-group asserts 
that (\cite{W2},  Thm.~14.10.5)
\begin{equation}\label{abs PT}  (\pi, \Hc) \simeq \Bigg( \int_\Lambda^\oplus \pi_\lambda \ d\mu(\lambda), \int_\Lambda^\oplus \Hc_\lambda \ d\mu(\lambda)\Bigg) \end{equation} 
that is
\begin{itemize}
\item $\Hc$ is unitarily isomorphic to a direct integral of Hilbert spaces $\int_\Lambda^\oplus \Hc_\lambda \ d\mu(\lambda)$, 
\item $(\pi_\lambda, \Hc_\lambda)$ is a unitary irreducible representation of $G$ for each $\lambda\in \Lambda$, 
\item Under the isomorphism $\Hc\simeq \int_\Lambda^\oplus \Hc_\lambda \ d\mu(\lambda)$ the representation
$\pi$ "diagonalizes" to  $\int_\Lambda^\oplus \pi_\lambda \ d\mu(\lambda)$. 
\end{itemize} 
Further we note that the measure class of $\mu$ in (\ref{abs PT}) is unique. 
\par Usually one prefers to combine equivalent representations 
and write (\ref{abs PT}) in the form 
\begin{equation} \label{abs PT2} (\pi, \Hc) \simeq \Bigg( \int_{\hat G}^\oplus  \rho\otimes {\rm id} \ d\mu(\rho), \int_{\hat G}^\oplus \Hc_\rho\otimes {\mathcal M}_\rho \ d\mu(\rho)
\Bigg) \end{equation} 
where $\mu$ is a measure on the unitary dual $\hat G$ and ${\mathcal M}_\rho$ is a 
Hilbert space which represents the multiplicity of
$\rho$ in $\pi$. 

\par We refer to (\ref{abs PT}) or (\ref{abs PT2}) as the abstract Plancherel theorem of the unitary representation 
$(\pi,\Hc)$.

\begin{ex} Consider $G=(\R,+)$. Then all unitary irreducible representation are one-dimensional and 
$$\R\to \hat G, \ \ \lambda\mapsto \chi_\lambda; \   \chi_\lambda(x) =e^{i\lambda x}$$ 
identifies $\hat G$ with $\R$.  The regular representation $(L, L^2(\R))$ then decomposes as 
$$(L, L^2(\R) ) \simeq \Bigg( \int_\R ^\oplus  \chi_\lambda \ d\lambda , \int_\R^\oplus \C_\lambda \ d\lambda\Bigg)$$
with the isomorphism given by the Fourier-transform. 
\end{ex} 

\subsection{Abstract Plancherel theory for $L^2(Z)$}\label{subsect abs} 

Our concern in these notes is with the unitary representation $(L, L^2(Z))$ of $G$ for a real spherical space
$Z$. The final objective is to establish a concrete Plancherel-theorem, that is an explicit determination of 
the measure class of $\mu$ together with the corresponding multiplicity Hilbert spaces $\mathcal M_\rho$. 
This however, is still an open problem. 

The map 
\begin{equation}\label{abstract F-t}
\F: L^2(Z) \to \int_{\Lambda}^\oplus \Hc_\lambda \ d\mu(\lambda) 
\end{equation}
is called the Fourier-transform.   Let us try to understand the corresponding multiplicity 
spaces.  For that we note that the Fourier inclusion $C_c^\infty(Z)\to \int_{\Lambda}^\oplus \Hc_\lambda \ d\mu(\lambda)$ is pointwise defined (see Exercise \ref{weighted Sobolev}). Hence we obtain 
for all $\lambda\in\Lambda$  (we ignore $\Lambda_0$) a $G$-morphism 
$$ T_\lambda: C_c^\infty(Z) \to \Hc_\lambda\, .$$
Dualizing we obtain an antilinear  $G$-morphism 
$$S_\lambda :  \Hc_\lambda \to {\mathcal D} (Z)$$
with ${\mathcal D}(Z)$ the distributions. Now it follows from Lemma \ref{dual Frobenius} 
that $S_\lambda$ induces an 
antilinear $G$-morphism  $S_\lambda^\infty : \Hc_\lambda^\infty \to C^\infty(Z)$.
In particular, 
$$\eta_\lambda: \Hc_\lambda^\infty \to \C, \ \ v \mapsto   \oline{S_\lambda^\infty(v)(z_0)}$$
yields a continuous $H$-equivariant linear functional.

\par To move on we need to describe the fibers of the map 
$$\Phi: \Lambda\to \hat G ,\ \  \lambda\mapsto [\pi_\lambda]\, .$$ For that we recall from (\ref{abs PT2})
the multiplicity Hilbert space ${\mathcal M}_{\rho}$ for $\rho \in \hat G$. 
For every $\rho \in \hat G$ we let $(v_j^\rho)_{1\leq j \leq n_\rho}$ be an orthonormal basis of ${\mathcal M}_\rho$ where $n_\rho\in \N_0\cup\{\infty\}$. 
Then $\Phi^{-1}(\rho) \simeq N_\rho:=\{ v^\rho_1,  \ldots, v_{n_\rho}^\rho\}$.
Moreover, outside a set of measure zero, the assignment
$$\Phi^{-1}(\rho) \ni\lambda \mapsto \eta_\lambda \in (\Hc_\lambda^{-\infty})^H$$
extends to an injective linear map
$\Span_\C N_\rho \hookrightarrow (\Hc_\lambda^{-\infty})^H$
(essentially a consequence of the fact that   $C_c^\infty(Z)$ is dense in $L^2(Z)$). 
We recall from Theorem \ref{KS-thm} that $\dim (\Hc_\rho^{-\infty})^H<\infty$ and obtain 
for the multiplicity version  (\ref{abs PT2})
of the abstract Plancherel-theorem that  $\dim {\mathcal M}_\rho<\infty$ and,
as a vector space,
$$ {\mathcal M}_\rho\subset   (\Hc_\rho^{-\infty})^H \qquad (\rho \in\hat G\bs \hat G_0)$$
where $\hat G_0=\Lambda_0$ is a subset of measure zero in $\hat G=\Lambda$.

We also obtain a formula for the inverse Fourier-transform. For $\phi\in C_c^\infty(Z)$ with $\F(\phi) = (v_\lambda)_{\lambda}$
one has 
\begin{equation} \label{F-inv} \phi =\int_\Lambda m_{v_\lambda, \eta_\lambda} \ d\mu(\lambda) \end{equation} 
understood as an identity of distributions, that is for all $\psi\in C_c^\infty(Z)$ one has 
$$\la \phi,\psi\ra_{L^2(Z)} = \int_\Lambda \la m_{v_\lambda, \eta_\lambda}, \psi\ra_{L^2(Z)} \ d\mu(\lambda)  \,.$$ 
Indeed, let $\F(\psi)=(w_\lambda)_\lambda$ and observe that  
$m_{v_\lambda,\eta_\lambda} =
\oline{{S_\lambda^\infty}(v_\lambda)}$ 
by  the definition of $\eta_\lambda$. Hence 
\begin{eqnarray*}& &  \int_\Lambda  \la m_{v_\lambda, \eta_\lambda}, \psi\ra_{L^2(Z)}  \ d\mu(\lambda)=
\int_\Lambda  \la \oline{S_\lambda(v_\lambda)}, \psi\ra_{L^2(Z)}  \ d\mu(\lambda)\\
& &\quad= \int_\Lambda  \la v_\lambda , T_\lambda(\psi)\ra_\lambda\ d\mu(\lambda)
=\int_\Lambda  \la v_\lambda , w_\lambda\ra_\lambda\ d\mu(\lambda) \end{eqnarray*}
which shows absolute convergence.

\section{The Schwartz space of a real spherical space}

\subsection{The polar decomposition of a real spherical space}\label{subsection polar}

Let $Z=G/H$ be a real spherical space.  We let $P<G$ be a minimal parabolic such that $PH$ is open and 
let $Q=LU\supset P$ be the associated $Z$-adapted parabolic subgroup. 

\par Recall the connected split torus $A\subset L$ and set $A_H:=A\cap H$.  Set $A_Z:=A/A_H$. 
The number 
$$\rank_\R Z:= \dim_\R \af_Z$$
is an invariant of $Z$ and called the {\it real rank of $Z$}.

\par Let $d:=\dim_\R\hf$ and write $\Gr_d(\gf)$ for the Grassmannian of $d$-dimensional subspaces of $\gf$. 
Set $\hf_{\lim}:= \lf \cap \hf +\oline\uf$
(previously denoted $\oline\qf_1$). 

We define an open cone $\af_Z^{--}$ in $\af_Z$ by the property: $X\in \af_Z^{--}$ if and only if 
$$\lim_{t\to \infty} e^{t\ad X} \hf = \hf_{\rm lim}$$
in $\Gr_d(\gf)$.  The closure of $\af_Z^{--}$ is denoted by $\af_Z^-$ and called the {\it compression cone}. 
Set $A_Z^-:=\exp(\af_Z^-)\subset A_Z$.

\begin{rmk} The compression cone has the following universal property: Let $(\pi,V)$ be a finite dimensional 
irreducible representation of $G$ which is $H$-semi-spherical, that is, there exists a vector $0\neq v_H\in V$ such that 
$\pi(h) v =\chi(h) v$ for a character of $H$.  Assume in addition that $V$ is $P$-semi-spherical and let $v_0\in V$ 
be a lowest weight vector.  Then $X\in \af_Z^{--}$ if and only if for all $(H,P)$-semi-spherical $(\pi,V)$ one has
$$ \lim_{t\to\infty} [\pi(\exp(tX))v_H]=[v_0]$$  
as limit in ${\mathbb P}(V)$ (see \cite{KKSS}, Sect. 5 for more on this topic). 
\end{rmk}

\par Notice that $A_Z=A/A_H$ can be naturally identified with a subset of $Z$ via $aA_H \mapsto a\cdot z_0$. 
Let $A_\C=\exp(\af_\C)<G_\C$ and define $A_{Z,\C}= A_\C/ A_\C \cap H_\C$.  We view 
$A_{Z,\C}$ as a subset of $Z_\C$. Note that $A_{Z,\C}\cap Z$ falls into finitely many 
$A_Z$-orbits and we let $\W$ be a set of representatives from $\exp(i\af_Z)\cdot z_0$.

\begin{theorem}\label{polar dec} 
{\rm (Polar decomposition  \cite{KKSS})} There exists a compact set $\Omega\subset G$
such that 
$$ Z= \Omega A_Z^- \W\cdot z_0\, .$$
Moreover, there is a finite set $F\subset G$ such that 
the above holds for $\Omega=FK$. 
\end{theorem} 

The polar decomposition tells us that the large scale geometry of $Z$ is determined by $A_Z^- \times \W$.

\subsection{The canonical weights on $Z$}

Assume now that $Z=G/H$ is unimodular. Let $\rho_\uf\in\af^*$ be defined as $\rho_\uf:={1\over 2} \tr \ad_\uf$. 
Observe that $Z$ unimodular implies that $\rho_\uf$ factors through a functional on $\af_Z$ (see \cite{KKSS2}, Lemma 4.2). 

Having defined the polar decomposition we can now give sharp bounds on the volume weight. 

\begin{prop} {\rm (\cite{KKSS2}, Prop. 4.3)}  Let $Z=\Omega A_Z^- \W\cdot z_0$.  Then there exist constants
$C_1, C_2>0$ such that for all $z=\omega aw\cdot z_0\in Z$ with $\omega\in \Omega$, $a\in A_Z^-$ and 
$w\in \W$ one has 
$$  C_1 \cdot a^{-2\rho_\uf} \leq \v(z) \leq C_2 \cdot a^{-2\rho_\uf}\, .$$
\end{prop}

Together with the volume weight there is the {\it radial weight function} on  $Z$ which is defined as 
$$ {\mathbf r}(z) := \sup_{z=\omega aw\cdot z_0\atop 
\omega\in \Omega, a\in A_Z^-, w\in \W} \|\log a\|\, $$
 for $\Omega\subset G$ a sufficiently large compact set.
This function is a weight on $Z$ by \cite{KKSS2}, Prop. 3.4.
 Later we need the following property.

\begin{lemma}\label{r-v-integral}
Let $s>\dim\af_Z$. Then
\begin{equation*}
\int_Z (1+\mathbf r(z))^{-s}\, \v(z)^{-1}\,dz < \infty .
\end{equation*} 
\end{lemma}

\begin{proof}
Let $\Gamma\subset A_Z$ be a lattice, that is $\log \Gamma\subset \af_Z $ is a 
lattice in the vector 
space $\af_Z$. 
Then 
\begin{equation} \label{summation} 
\sum_{a\in \Gamma } ( 1+ \|\log a\|)^{-s}<\infty.
\end{equation}
It follows from Theorem \ref{polar dec} that
we can assume (after enlarging $\Omega$) 
$
Z = \Omega \Gamma^- \W \cdot z_0\, 
$ 
where $\Gamma^-:=\Gamma \cap A_Z^-$. 
With this we have 
\begin{equation*}
\int_Z (1+\mathbf r(z))^{-s}\, \v(z)^{-1}\,dz\, \le \sum_{y\in \Gamma^-\W}
\int_{\Omega y\cdot z_0}(1+\mathbf r(z))^{-s}\, \v(z)^{-1}\,dz .
\end{equation*}
By using (\ref{weight below}) we find positive constants such that
$$\mathbf r(z)\ge c_1 \mathbf r(y\cdot z_0), \qquad \v(z)\ge c_2\v(y\cdot z_0),$$
for all $z\in \Omega y\cdot z_0$, and hence
$$\int_{\Omega y\cdot z_0}(1+\mathbf r(z))^{-s}\, \v(z)^{-1}\,dz 
\leq C (1+\mathbf r(y\cdot z_0))^{-s} \frac{\vol_Z(\Omega y\cdot z_0)}{\v(y\cdot z_0)} .$$
The ratio of volumes is bounded (see Exercise \ref{weight exercise}),
and hence the lemma follows from the definition of $\mathbf r$ and 
(\ref{summation}).
\end{proof}

\subsection{Definition and basic properties of the Schwartz space} 

Fix a basis $X_1,\ldots, X_n$ of $\gf$. For $\alpha\in\N_0^n$ set $X^\alpha:=X_1^{\alpha_1}\cdot \ldots\cdot  X_n^{\alpha_n}$. 
For a test function $f\in C_c^\infty(Z)$ and $m,k\in\N_0$ we define norms 

$$p_{m,k} (f):=\sup_{z\in Z} ( 1 + {\mathbf r}(z))^m \,\v(z)^ {\frac12} \sum_{|\alpha|\leq k } |dL(X^\alpha) f (z)|$$

$$q_{m,k}(f):= \Big(\sum_{|\alpha|\leq k} \| (1 +{\mathbf r})^m dL(X^\alpha) f \|_{L^2(Z)}^2 \Big)^{1\over 2} $$

For each $k\in\N_0$ we
denote by $\Cc^k_m(Z)$ and $\Lc_m^{2,k}(Z)$ the completion of $C_c^\infty (Z)$ with respect to 
$p_{m,k}$ and $q_{m,k}$, respectively.  
Note that $\Lc^{2,k}_m(Z)$ is a Hilbert space, more precisely, a weighted 
Sobolev space. 

In this context we record (see \cite{KKSS2}, Prop. 5.1): 

\begin{lemma} \label{semi-norms}
The two families $(p_{m,k})_{m,k}$ and $(q_{m,k})_{m,k}$ are equivalent, i.e.~they 
define the same 
topology on $C_c^\infty(Z)$. 
\end{lemma}

\begin{proof}
Let $l>\dim G/2$, then for $f\in  C_c^\infty(Z)$
$$|f(z)|\leq C_B \v(z)^{-\frac12} \| f\|_{2,l,Bz}$$
by Lemma \ref{S-lemma}. By applying (\ref{weight below}) to the weight
$\mathbf r$ we easily see that
\begin{equation*}
( 1 + {\mathbf r}(z))^m \| f \|_{2,l,Bz}\leq 
C q_{m,l}(f)
\end{equation*}
for all $z$, all $f\in  C_c^\infty(Z)$, and some constant $C>0$. Hence 
\begin{equation}\label{eval bound by q}
( 1 + {\mathbf r}(z))^m \, \v(z)^{\frac12} |f(z)|\leq 
C_BC q_{m,l}(f)
\end{equation}
and it follows that $p_{m,k}$ is dominated by
$q_{m,k+l}.$
For the converse domination we use
$$
q_{m,k}(f)^2 \leq \Big(\int_Z ( 1 + {\mathbf r}(z))^{2(m-n)} \v(z)^{-1} \,dz\Big)\cdot p_{n,k}(f)^2
$$
and  Lemma \ref{r-v-integral}.
\end{proof}

This gives  us in particular that 
$$\Cc(Z):=\bigcap_{m,k} \Cc^k_m(Z) =\bigcap_{m,k} \Lc^{2,k}_m (Z). $$ 
We call $\Cc(Z)$ the {\it Harish-Chandra Schwartz-space of $Z$.}  Note that 
$$C_c^\infty (Z)\subset \Cc(Z) \subset C^\infty(Z)\cap L^2(Z)\, .$$
Building on the general theory developed in \cite{B} it was shown in \cite{KKSS2}, Prop. 5.2:

\begin{prop} Let $m>2\rank_\R Z$ and $l>{\dim \gf\over 2}$. Then the inclusion $\Lc^{2,l}_m(Z)\to L^2(Z)$ 
is Hilbert-Schmidt.  In particular (by Theorem \ref{G-K}), the inclusion map $\Cc(Z) \to L^2(Z)$ 
composed with the Fourier transform (\ref{abstract F-t}) is pointwise defined.
\end{prop}

\begin{proof}
Let $\{\xi_n\}_{n\in\N}$ be an orthonormal basis for $\Lc^{2,k}_m(Z)$. We must show
$$\sum_n \| \xi_n \|^2_{L^2(Z)} <\infty .$$ 
By the Sobolev Lemma \ref{S-lemma} the evaluation at $z\in Z$ is a continuous
linear form on $\Lc^{2,l}_m(Z)$. Hence there exists for each $z\in Z$ a function
$\kappa_z\in \Lc^{2,l}_m(Z)$ such that $f(z)=\la f, \kappa_z\ra$ for all $f$,
and it follows from (\ref{eval bound by q}) that
$$q_{m,l}(\kappa_z)\le C_BC(1+\mathbf r(z))^{-m} \v(z)^ {-\frac12}.$$
Hence 
$$\sum_n \int_Z | \xi_n(z)|^2 \,dz = \int_Z\sum_n |\la\xi_n,\kappa_z\ra|^2 \,dz
=\int_Z q_{m,l}(\kappa_z)^2\, dz<\infty$$
by  Lemma \ref{r-v-integral}.
\end{proof}

\section{The  notion of a tempered representation}

We recall the abstract Plancherel theorem for the regular representation:
$$L^2(G/H)\simeq \int_{\hat G} \Hc_\pi \otimes {\mathcal M}_\pi  \ d \mu (\pi)$$
with the finite dimensional multiplicity spaces ${\mathcal M}_\pi\subset (\Hc_\pi^{-\infty})^H$.  

We wish to give a necessary condition for a spherical pair $(V,\eta)$ to contribute to 
this decomposition of $L^2(Z)$, 
i.e.~there exists a unitary representation $\pi \in \supp(\mu)$, for which 
the underlying Harish-Chandra module $V_\pi$ equals $V$, 
and an $\eta_\pi\in {\mathcal M}_\pi$ which is equal to $\eta$.
Once the theory is developed further it is expected that the condition below is also necessary.

\begin{definition} Let $V$ be a Harish-Chandra module and let $\eta\in (V^{-\infty})^H$.
The spherical pair $(V,\eta)$ is called 
{\rm tempered} if there exists $m\in \Z$ such that 
$$ \sup_{z\in Z} |m_{v,\eta} (z) |\sqrt{\v(z)}   ( 1+ {\mathbf r}(z))^m <\infty \qquad (v\in V^\infty) \, .$$ 
\end{definition}

\begin{prop}[\cite{KKSS2}, Prop.~5.5]
There exists a subset $\hat G_0\subset \hat G$ of $\mu$-measure zero such that 
$(V_\pi, \eta_\pi)$ is tempered for all $\pi\in\hat G\bs \hat G_0$ and
$\eta_\pi \in {\mathcal M}_\pi$. 
\end{prop}

\begin{proof}   The argument is parallel to the treatment in Subsection \ref{subsect abs} with $\Lc_m^{2,k}$ instead of 
$C_c^\infty(Z)$.  We start with a continuous morphism ${\mathcal T}_\pi:  \Lc_m^{2,k} \to \Hc_\pi$ and take its dual
${\mathcal T}_\pi':  \Hc_\pi' \to (\Lc_m^{2,k})'$.  Now observe that there are natural 
$G$-equivariant anti-linear isomorphisms
$\Hc_\pi' \simeq \Hc_\pi$  and $(\Lc_m^{2,k})'\simeq \Lc_{-m}^{2, -k}$.  Hence we obtain a 
linear
$G$-morphism 
$$ {\mathcal S}_\pi : \Hc_\pi \to \Lc_{-m}^{2, -k}\, .$$
Taking smooth vectors then yields a $G$-morphism  ${\mathcal S}_\pi^\infty: \Hc_\pi^\infty \to  \Lc_{-m}^{2, \infty}$.
In analogy to Subsection \ref{subsect abs} we have that ${\mathcal S}_\pi^\infty(v)=m_{v,\eta_\pi}$ 
for all $v\in 
\Hc_\pi^\infty$.

Finally we observe that $q_{-m,k+l}$ dominates $p_{-m,k}$  by the same argument as in the proof of
Lemma \ref{semi-norms}. \end{proof}

\end{document}